\documentclass[reqno,english]{amsart}

\usepackage{amsmath,amssymb,verbatim}
\usepackage{amsthm}
\usepackage{mathtools}
\usepackage[shortlabels]{enumitem}
\usepackage[bookmarks]{hyperref}
\usepackage[capitalize]{cleveref}
\usepackage{tikz}
\usetikzlibrary{patterns}
\usetikzlibrary{decorations.pathmorphing}
\usepackage{caption}
\usepackage{siunitx}
\usepackage{stmaryrd}
\usepackage{mathrsfs}
\usepackage{csquotes}
\usepackage[margin=1.5in]{geometry}
\usepackage{nicefrac}
\usepackage{multicol}
\usepackage{pict2e,picture}
\usepackage{bussproofs}

\newtheorem{thm}{Theorem}[section]
\newtheorem{prop}[thm]{Proposition}
\newtheorem{lem}[thm]{Lemma}
\newtheorem{cor}[thm]{Corollary}
\newtheorem{fact}[thm]{Fact}

\newtheorem{quest}[thm]{Question}
\theoremstyle{definition}
\newtheorem{defn}[thm]{Definition}
 
\theoremstyle{remark}

\newcommand{\Lc}{\mathcal{L}}

\newcommand{\Pc}{\mathcal{P}}

\newcommand{\Tc}{\mathcal{T}}

\newcommand{\Toot}{\Leftrightarrow}
\newcommand{\toot}{\leftrightarrow}
\newcommand{\To}{\Rightarrow}

\newcommand{\sqin}{%
  \mathrel{\vphantom{\sqsubset}\text{%
      \mathsurround=0pt
      \ooalign{$\sqsubset$\cr$-$\cr}%
    }}%
}

\newcommand{\Rb}{\mathbb{R}}
\newcommand{\Nb}{\mathbb{N}}
\newcommand{\Qb}{\mathbb{Q}}
\newcommand{\Zb}{\mathbb{Z}}

\newcommand{\xbar}{\bar{x}}
\newcommand{\abar}{\bar{a}}
\newcommand{\bbar}{\bar{b}}

\newcommand{\ybar}{\bar{y}}
\newcommand{\zbar}{\bar{z}}
\newcommand{\fbar}{\bar{f}}
\newcommand{\cbar}{\bar{c}}
\newcommand{\gbar}{\bar{g}}

\newcommand{\e}{\varepsilon}

\providecommand{\dotplus}{%
  \mathbin{%
    \vphantom{+}%
    \text{%
      \mathsurround=0pt %
      \ooalign{%
        \noalign{\kern-.35ex}%
        \hidewidth$\smash{\cdot}$\hidewidth\cr %
        \noalign{\kern.35ex}%
        $+$\cr %
      }%
    }%
  }%
}

\pgfdeclareradialshading{ring}{\pgfpoint{0cm}{0cm}}%
{rgb(0cm)=(0,0,0);
  rgb(0.1cm)=(0,0,0); %
  rgb(0.2cm)=(0.0,0.0,0.0);
  rgb(0.3cm)=(0.0,0.0,0.0);
  rgb(0.4cm)=(0.2,0.2,0.2);
  rgb(0.5cm)=(0.4,0.4,0.4);
  rgb(0.6cm)=(0.6,0.6,0.6);
  rgb(0.7cm)=(0.8,0.8,0.8);
  rgb(0.8cm)=(1.0,1.0,1.0);
  rgb(0.9cm)=(1.0,1.0,1.0);
  rgb(1.0cm)=(1,1,1)}

\newcommand{\MSE}{\mathsf{MSE}}%

\newcommand{\GPK}{\mathsf{GPK}}
\newcommand{\ZF}{\mathsf{ZF}}

\newcommand{\PA}{\mathsf{PA}}

\newcommand{\NF}{\mathsf{NF}}
\newcommand{\NFU}{\mathsf{NFU}}

\newcommand{\TSTI}{\mathsf{TSTI}}
\newcommand{\CL}{\mathrm{C\L}}

\newcommand{\res}{{\upharpoonright}}

\makeatletter \DeclareRobustCommand{\cset}{\@ifstar\star@cset\normal@cset}
\newcommand{\star@cset}[1]{\left\llbracket#1\right\rrbracket}
\newcommand{\normal@cset}[2][]{\mathopen{#1\llbracket}#2\mathclose{#1\rrbracket}}
\makeatother

\DeclareMathOperator{\Card}{Card}
\DeclareMathOperator{\otp}{otp}
\DeclareMathOperator{\otok}{otok}
\DeclareMathOperator{\ctok}{ctok}

\DeclareMathOperator{\Ord}{Ord}
\DeclareMathOperator{\ord}{ord}

\newcommand{\HH}{\mathrm{H}}

\newcommand{\Inf}{\mathsf{Inf}}

\DeclareMathOperator{\st}{st}

\newcommand{\naive}{na\"\i ve}

\DeclareMathOperator{\tc}{tc}

\newcommand{\ol}{\overline}

\DeclareMathOperator{\dist}{dist}

\newcommand{\Lsq}{\Lc_{\sqin}}

\makeatletter %
\newcommand*\superimpose[2]{%
  \ooalign{$\m@th#1\@firstoftwo#2$\cr
    \hidewidth$\m@th#1\@secondoftwo#2$\hidewidth}%
}
\makeatother %

\makeatletter
\newcommand{\dashedeq}[1]{\mathrel{\vphantom{\geq}\mathpalette\dashed@eq{#1}}}
\newcommand{\dashed@eq}[2]{%
  \vcenter{%
    \offinterlineskip
    \roundcap
    \linethickness{1.2\dimexpr\variable@rule{#1}\relax}%
    \sbox\z@{$\m@th#1#2$}%
    \setlength{\unitlength}{\dimexpr(\wd\z@-2\dimexpr\dashed@eq@kern{#1})/8}%
    \sbox\tw@{\begin{picture}(1,0)\Line(0,0)(1,0)\end{picture}}%
    \ialign{%
      ##\cr
      \copy\z@\cr
      \noalign{\vskip 0.3\ht\z@}
      \kern\dashed@eq@kern{#1}%
      \copy\tw@\hfil\copy\tw@\hfil\copy\tw@\hfil\copy\tw@
      \kern\dashed@eq@kern{#1}%
      \cr
    }%
  }%
}
\newcommand{\dashed@eq@kern}[1]{%
  \ifx#1\displaystyle 0.12\wd\z@\else
    \ifx#1\textstyle 0.12\wd\z@\else
      \ifx#1\scriptstyle 0.14\wd\z@\else
        0.16\wd\z@
  \fi\fi\fi
}
\newcommand{\variable@rule}[1]{%
  \fontdimen8  
  \ifx#1\displaystyle\textfont3\else
    \ifx#1\textstyle\textfont3\else
      \ifx#1\scriptstyle\scriptfont3\else
        \scriptscriptfont3\relax
  \fi\fi\fi
}
\makeatother

\newcommand{\exc}[4]{\left[#1 : #2 < #3\sim #4\right]}%

\DeclareMathOperator{\dis}{dis}

\DeclareMathOperator{\succc}{succ}
\DeclareMathOperator{\ind}{ind}

\providecommand{\dotdiv}{%
  \mathbin{%
    \vphantom{+}%
    \text{%
      \mathsurround=0pt %
      \ooalign{%
        \noalign{\kern-.35ex}%
        \hidewidth$\smash{\cdot}$\hidewidth\cr %
        \noalign{\kern.35ex}%
        $-$\cr %
      }%
    }%
  }%
}

\DeclareMathOperator*{\qqq}{qnt}

\newcommand{\opand}{\mathbin{\&}}

\DeclareMathOperator{\chn}{chn}

\DeclareMathOperator{\tp}{tp}

\newcommand{\rin}{\mathrel{\hat{\epsilon}}}

\newcommand{\LL}{\text{\upshape\L}}

\begin{document}

\title{A metric set theory with a universal set}
\address{Department of Mathematics\\
  University of Maryland\\
  College Park, MD 20742, USA}
\author{James Hanson}
\email{jhanson9@umd.edu}
\date{\today}

\keywords{metric set theory}
\subjclass[2020]{03E70, 03B50, 03C66}

\begin{abstract}
  Motivated by ideas from the model theory of metric structures, we introduce a metric set theory, $\MSE$, which takes bounded quantification as primitive and consists of a natural metric extensionality axiom (the distance between two sets is the Hausdorff distance between their extensions) and an approximate, non-deterministic form of full comprehension (for any real-valued formula $\varphi(x,y)$, tuple of parameters $a$, and $r < s$, there is a set containing the class $\{x: \varphi(x,a) \leq r\}$ and contained in the class $\{x:\varphi(x,a) < s\}$).  We show that $\MSE$ is sufficient to develop classical mathematics after the addition of an appropriate axiom of infinity. We then construct canonical representatives of well-order types and prove that ultrametric models of $\MSE$ always contain externally ill-founded ordinals, conjecturing that this is true of all models. To establish several independence results and, in particular, consistency, we construct a variety of models, including pseudo-finite models and models containing arbitrarily large standard ordinals. Finally, we discuss how to formalize $\MSE$ in either continuous logic or \L ukasiewicz logic.
\end{abstract}

\maketitle

\section*{Introduction}
\label{sec:intro}

Ever since the discovery of the inconsistency of full comprehension principles at the turn of the last century, there have been various efforts to rescue the idea and formulate systems in which the entire domain of discourse is meaningfully represented as an element of the domain of discourse itself.

Neo-\naive\ set theories commonly take one of two approaches to repairing full comprehension. One is to weaken the comprehension principle while maintaining full classical logic, and the other is to weaken the underlying logic while maintaining the full comprehension principle. Extensionality is often weakened or abandoned entirely. While there have been many investigations into such theories, there are approximately three in particular we will be occasionally comparing to ours: Quine's New Foundations, $\NF$, and Jensen's ``slight\ (?)\ modification'' thereof, $\NFU$; the positive topological set theory $\GPK^+$, studied most prominently by Esser; and Cantor-\L ukasiewicz set theory, originally isolated by Skolem but named by H\'ajek, which consists simply of the full comprehension scheme interpreted in the $[0,1]$-valued \L ukasiewicz predicate logic. $\NF(\mathsf{U})$ and $\GPK^+$ fall under the first approach mentioned above, and Cantor-\L ukasiewicz set theory, abbreviated $\CL_0$ by H\'ajek, falls under the second. $\NF$ and $\GPK^+$ have full extensionality, but $\CL_0$ is entirely inconsistent with it and $\NFU$ weakens it by allowing urelements. To keep this introduction short, we will point the reader to \cite{sep-settheory-alternative} for an overview of $\NF(\mathsf{U})$ and $\GPK^+$ and to \cite[Ch.\ 4.5]{Hajek-2015} for an overview of Cantor-\L ukasiewicz set theory. In particular, H\'ajek's result \cite[Th.\ 4.17]{Hajek-2015} that $\CL_0$ has no $\omega$-models will be relevant, in that we will find a similar real-valued failure of induction in our theory (\cref{thm:ord-ill}), although for us the failure may occur at arbitrarily large ordinals (\cref{thm:MSE-consistent}), rather than necessarily at $\omega$.

In this paper, we introduce a new set theory, $\MSE$, which takes a combined approach to repairing comprehension: We weaken the comprehension scheme \emph{and} weaken (or, more charitably, generalize) the underlying logic by working in a real-valued logic. For the sake of presentation, we will work in a slightly bespoke formalism, heavily based on first-order continuous logic, that capitalizes on the fact that our theory is a theory of sets that are actually sets of elements, rather than arbitrary $[0,1]$-valued predicates on some domain. In particular, our formalism will take bounded quantification as primitive, which is not possible to do in continuous logic, introduced in its modern form in \cite{MTFMS}.

Models of $\MSE$ are triples $(M,d,\sqin)$, where $(M,d)$ is a complete metric space and ${\sqin} \subseteq M^2$ is a closed binary relation. Such a structure is a model of $\MSE$ if it satisfies a strong metric form of extensionality and a weak approximation of comprehension. The strong form of extensionality, which we refer to as \emph{$\HH$-extensionality}, requires that for any $a,b \in M$, $d(a,b)=d_{\HH}(\{x : x \sqin a\}, \{x : x \sqin b\})$, where $d_{\HH}$ is the Hausdorff metric on sets (\cref{defn:H-ext}). We refer to $\HH$-extensional structures $(M,d,\sqin)$ as \emph{metric set structures}.

  The weak form of comprehension is the following principle: For any real-valued formula $\varphi(x,\ybar)$, any real numbers $r < s$, and any tuple of parameters $\abar \in M$, there is a set $b \in M$ such that for any $c$, if $\varphi(c,\abar) \leq r$, then $c \sqin b$ and if $c \sqin b$, then $\varphi(c,\abar) < s$ (\cref{defn:exc}). Crucially, we make no guarantees about membership of those $c$'s for which $\varphi(c,\abar)$ falls in the gap between $r$ and $s$ and so in this sense the principle is non-determinisitc.

  The word that we find most accurately captures this principle is \emph{excision}, the idea being that we are only able to cut out a desired set somewhat crudely. From this we get the initialism $\MSE$, for \emph{Metric Sets with Excision}. Of course the nature of this principle depends entirely on what is meant by `real-valued formula,' which is formalized in \cref{sec:formulas}, but the crucial fact is that these formulas are automatically uniformly continuous with regards to the metric. In particular, we have no direct access to the relation $\sqin$ as a $\{0,1\}$-valued predicate and instead can only use it in instances of bounded quantification, such as $\inf_{y \sqin z}\varphi(x,y)$.

After defining $\MSE$ and developing some techniques for constructing particular sets, we will establish that $\MSE$ is sufficiently strong and expressive by showing that it (with an axiom of infinity) interprets classical $\TSTI$\footnote{See \cite[Sec.~1.1.2]{EForster1992-EFOSTW} for the definition of $\TSTI$.} (or, equivalently, full $\omega$th-order arithmetic or the theory of a Boolean topos with a natural numbers object), which is well known to be more than sufficient for everyday mathematics. In particular, we do this by considering \emph{uniformly discrete sets}, which are better behaved than arbitrary sets in models of $\MSE$. We then build canonical representatives of internal well-order types in models of $\MSE$ (which we call ordinals). %
We show that $\MSE$ has no ultrametric $\beta$-models (i.e., models that are correct about well-foundedness) by showing that the class of ordinals of any such model $M$ admits an external map $s$ to $(0,1]$ that is non-increasing and has dense image (implying that the preimage of $(0,1)$ under $s$ has no least element). When we eventually construct models of $\MSE$ by using a non-standard modification of the standard construction of models of $\GPK^+$, we show that they can have arbitrarily large standard ordinals (\cref{thm:MSE-consistent}). This is of course similar to the situation with $\NFU$, which has no true $\beta$-models yet can have arbitrarily large well-founded parts, but the mode of failure is more conceptually similar to the mechanism that prevents $\CL_0$ from having $\omega$-models in that it involves the difficulty of robustly formalizing induction for real-valued predicates. We also construct a pseudo-finite model of our theory (without infinity). This establishes that $\MSE$ without infinity has an incredibly low consistency strength, lower than Robinson arithmetic, in contrast to $\NFU$ and $\GPK^+$ without infinity.\footnote{$\NFU$'s consistency strength is strictly between Robinson arithmetic and $\PA$. $\GPK^+$ is equiconsistent with full second-order arithmetic.} We also show that any complete metric space of diameter at most $1$ can be embedded as an internal set of Quine atoms in a model of $\MSE$, which in particular shows that not all models of $\MSE$ are ultrametric. Nevertheless, all models we are able to construct admit the map $s$ as before, so we conjecture that this is in fact always the case.

Finally, we show how to formalize our theory in either continuous logic or \L ukasiewicz predicate logic. In these contexts, we consider structures (without a given metric) of the form $(M,e)$, where $e$ is a binary $[0,1]$-valued predicate on $M$. The intended interpretation of $e(x,y)$ is the quantity $\inf_{z \sqin y}d(x,z)$. Our strong form of extensionality ensures that the metric $d(x,y)$ can be recovered from $e(x,y)$ by the formula $d_e(x,y)\coloneqq \inf_{z}|e(z,x)-e(z,y)|$. The $\HH$-extensionality axiom now takes the form
  \[
    \sup_{xy}|e(x,y)-\inf_z\min(d_e(x,z)+2e(z,y),1)| = 0,
  \]
  and the axiom scheme of excision consists of
  \[
    \sup_{\ybar}\inf_z \sup_x \max(\min(e(x,z),-\varphi(x,\ybar)),\min(\e_\varphi-e(x,z),\varphi(x,\ybar) - 1)) \leq 0
  \]
  for each restricted $\Lc_e$-formula $\varphi(x,\ybar)$, where $\e_\varphi$ is a certain rational number directly computable from $\varphi$. We show that models of the above theory are precisely pre-models of $\MSE$ in the sense that the completion with regards to $d_e$ yields a model of $\MSE$ (by taking $x\sqin y$ to be the relation $e(x,y) = 0$). Furthermore, all models of $\MSE$ arise in this way.

  In \L ukasiewicz logic, we use the predicate\footnote{The hat in $\rin$ is merely to help visually distinguish it from the four other epsilon-like symbols in this paper, $\in$, $\sqin$, $e$, and $\e$. We never use the symbol $\epsilon$, and $\rin$ is only used in the last section of the paper. $\in$ always refers to standard set-theoretic membership and $\e$ is always a real number.} $x \rin y$ instead of $e(x,y)$, with the intended meaning being that $(x \rin y) = 1-e(x,y)$. The $\HH$-extensionality axiom is directly translated as
  \[
    \forall x \forall y (x \rin y \toot \exists z(x=_ez \opand z \rin y \opand z \rin y)),
  \]
  where $x=_ey$ is the formula $\forall z(z \rin x \toot z \rin y)$ (which is the same as $1-d_e(x,y)$), and the axiom scheme of excision is shown to be equivalent to the scheme
  \[
    \forall \ybar \exists z \forall x (x \rin z \vee (\neg\varphi\opand\neg \varphi \opand \neg \varphi))\wedge ((\underbrace{\neg x\rin z \opand \cdots \opand \neg x \rin z}_{6\cdot\#\varphi~\text{times}}) \vee(\varphi \opand \varphi \opand \varphi))
  \]
  for each \L ukasiewicz formula $\varphi(x,\ybar)$, where $\# \varphi$ is the number of instances of $\rin$ in $\varphi$. %

\section{Specification of $\MSE$}
\label{sec:MSE-spec}

\subsection{$\HH$-extensionality and metric set structures}

The structures we will be considering will be of the form $(M,d,\sqin)$, where $(M,d)$ is a complete metric space, and ${\sqin} \subseteq M^2$ is a closed binary relation. As is suggested by the notation, $\sqin$ is meant to be interpreted as a set membership relation, and, as such, we would like for it to be extensional. Obviously we could just require extensionality of $\sqin$ as a binary relation in the standard sense, but for a few different reasons, we will opt to place a stronger condition on $\sqin$. To define this condition, recall one of many equivalent definitions of the \emph{Hausdorff distance} between subsets of a metric space:
\begin{defn}\label{defn:Haus-dist}
  The \emph{Hausdorff distance between $A$ and $B$}, written $d_{\HH}(A,B)$, is the unique smallest element of $[0,\infty]$ such that for any $r > d_{\HH}(A,B)$,
  \begin{itemize}
  \item for every $a \in A$, there is a $b \in B$ such that $d(a,b) < r$, and
  \item for every $b \in B$, there is an $a \in A$ such that $d(a,b) < r$.
  \end{itemize}
\end{defn}
On the full power set of $M$, $d_{\HH}$ is an extended pseudo-metric, but on the collection of close subsets of $M$, it is an extended metric. We will be concerned exclusively with $[0,1]$-valued metrics. In this context, it makes sense to modify the above definition to take $d_{\HH}$ to be $[0,1]$-valued as well. This only changes the distance between the empty set and non-empty sets. In particular, $d_{\HH}(\varnothing, A) = 1$ for any non-empty $A$. This is the definition of the Hausdorff distance we will actually use. 

The form of \cref{defn:Haus-dist} above makes it clear that $d_{\HH}$ is a direct metric generalization of extensional equality of sets. $A=B$ if and only if for every $a \in A$, there is a $b \in B$ such that $a=b$ and for every $b \in B$, there is an $a \in A$ such that $a=b$. In this way, we take as our extensionality axiom a direct translation of the statement `$A = B$ if and only if $A$ and $B$ are coextensive.'

\begin{defn}\label{defn:H-ext}
  Given a metric space $(M,d)$ and a binary relation ${\sqin} \subseteq M^2$, we say that $\sqin$ is \emph{$\HH$-extensional} if for any $a,b \in M$, $d(a,b) = d_{\HH}(\{x:x\sqin a\},\{x:x\sqin b\})$.

  We say that $(M,d,\sqin)$ is a \emph{metric set structure} if $(M,d)$ is a complete metric space, $d$ is $[0,1]$-valued, and ${\sqin}\subseteq M^2$ is closed and $\HH$-extensional.
\end{defn}

Metric set structures are a direct generalization of extensional digraphs (i.e., discrete models of the extensionality axiom). If $\delta$ is a $\{0,1\}$-valued metric on $V$, then $(V,\delta,E)$ is a metric set structure if and only if $(V,E)$ is an extensional digraph.

We should note that the definition of $\HH$-extensionality contains a somewhat arbitrary choice. After all, we could have just as easily required that $d(a,b) = (d_{\HH}(\{x:x\sqin a\},\allowbreak\{x:x\sqin b\}))^{\nicefrac{1}{2}}$. That said, the choice we have made is reasonable and seems to work well, so we have not investigated other possibilities in this paper. Moreover, this notion of extensionality appears in a previous paper of the author \cite[Def.~6.4]{Hanson2021AnalogR}. %

A commonly cited benefit of extensionality is that it allows one to take $\in$ as the only primitive notion, with $x=y$ being defined as $\forall z (z \in x \toot z \in y)$. It seems unlikely that we will be able to do something similar with $\sqin$, but we can do something similar with the natural $[0,1]$-valued version of $x\sqin Y$, which is the distance from $x$ to the elements of $Y$, commonly written $\dist(x,Y)$ or $d(x,Y)$. In the context of a set theory, $d(x,Y)$ is entirely unacceptable, being tantamount to writing $x = Y$ to mean $x \in Y$. $\dist(x,Y)$ is too long to use frequently, so we will introduce the following notation.

\begin{defn}
  We write $e(x,y)$ for $\inf\{d(x,z) : z \sqin y\}$.
\end{defn}

Another useful characterization of the Hausdorff metric is this: $d_{\HH}(A,B) = \sup_z|\dist(z,A)-\dist(z,B)|$. This means that if $(M,d,\sqin)$ is a metric set structure, we have that $d(x,y) = \sup_z|e(z,x)-e(z,y)|$. Since $x\sqin y$ if and only if $e(x,y) = 0$, it should be possible to take $e(x,y)$ as our only primitive notion. This is the approach we will take in the `official' continuous logic formulation of $\MSE$, which we will discuss in \cref{sec:formalizing-in-CL}.

\subsection{Formulas}
\label{sec:formulas}

Our formalism will be a small modification of first-order continuous logic, introduced in its modern form in \cite{MTFMS}. Our only predicate symbol will be the metric, $d(x,y)$, but we will take bound quantifiers of the form $\sup_{x \sqin y}$ as a primitive notion. Note though that we cannot access $x \sqin y$ as a formula directly. %

\begin{defn}\label{defn:Lsq-formulas}
  Our set of \emph{formulas}, written $\Lsq$, is the smallest non-empty set of expressions satisfying the following: For any $\varphi,\psi \in \Lsq$, variables $x$ and $y$, and $r \in \Rb$, $\Lsq$ contains the expressions
\begin{multicols}{2}
  \begin{itemize}
  \item $1$,
  \item $d(x,y)$,
  \item $\varphi + \psi$,
  \item $\max(\varphi,\psi)$,
  \item $\min(\varphi,\psi)$, 
  \item $r \cdot \varphi$, 
  \item $\sup_x\varphi$,
  \item $\inf_x\varphi$,
  \item $\sup_{x\sqin y}\varphi$, 
  \item $\inf_{x\sqin y}\varphi$. 
  \end{itemize}
\end{multicols}
When we need to be more specific, we will refer to elements of $\Lsq$ as \emph{$\Lsq$-formulas}.

We refer to quantifiers of the form $\sup_{x\sqin y}$ or $\inf_{x\sqin y}$ as \emph{bounded}. The \emph{free variables} of a formula $\varphi$ are defined in the obvious way. We write $\varphi(\xbar)$ to indicate that the free variables of $\varphi$ are included in $\xbar$. %

For bookkeeping purposes, we will need the inductively defined quantity given by
\begin{itemize}
\item $v(1) = v(d(x,y)) = 1$,
\item $v(\varphi + \psi)=v(\varphi)+v(\psi)$,
\item $v(\max(\varphi,\psi))=v(\min(\varphi,\psi))=\max(v(\varphi),v(\psi))$,
\item $v(r \cdot \varphi) = |r|v(\varphi)$, 
\item $v(\sup_x\varphi)=v(\inf_x\varphi)=v(\sup_{x\sqin y}\varphi)=v(\inf_{x\sqin y}\varphi)=v(\varphi)$.
\end{itemize}
\end{defn}

We are allowing ourself arbitrary real numbers in \cref{defn:Lsq-formulas} because it will be convenient in several places. This convenience comes at a cost later in \cref{sec:formalizing-in-CL}, however.

The intended interpretation of a given formula is clear, although we do need to specify the behavior of bounded quantifiers over empty sets. This is the first of two reasons why we defined the quantity $v(\varphi)$.

\begin{defn}\label{defn:formula-evaluate}
  Given a metric set structure $(M,d,\sqin)$ we define real-valued functions $\varphi^M$ for $\varphi \in \Lsq$ inductively:
  \begin{itemize}
  \item $1^M = 1$.
  \item $(d(a,b))^M = d(a,b)$ for all $a,b \in M$.
  \item $(\varphi+\psi)^M=\varphi^M+\psi^M$. We define $(\max(\varphi,\psi))^M$, $(\min(\varphi,\psi))^M$, and $(r\cdot \varphi)^M$ similarly.
  \item $(\sup_x\varphi(x,\abar))^M = \sup\{\varphi^M(b,\abar) : b \in M\}$. We define $(\inf_x\varphi(x,\abar))^M$ similarly.
  \item $(\sup_{x\sqin a}\varphi(x,a,\bbar))^M = \sup\{\varphi^M(c,a,\bbar) : c \sqin a\}$ if $c \sqin a$ for some $c \in M$. $(\inf_{x\sqin a}\varphi(x,a,\bbar))^M$ is defined similarly if $c \sqin a$ for some $c \in M$.
  \item $(\sup_{x\sqin}\varphi(x,a,\bbar))^M = -v(\varphi)$ if $c \not\sqin a$ for all $c \in M$.
  \item $(\inf_{x\sqin a}\varphi(x,a,\bbar))^M = v(\varphi)$ if $c \not\sqin a$ for all $c \in M$.
  \end{itemize}
We write expressions such as $M \models \varphi(\abar) \leq \psi(\bbar)$ to mean that $\varphi^M(\abar) \leq \psi^M(\bbar)$. We may also write expressions like $M \models \varphi(\abar) = r$. %
\end{defn}

The conventions regarding suprema and infima of empty sets were chosen so that formulas would always be real-valued (rather than taking on values in $\Rb\cup\{\pm \infty\}$) and so that $\sup$ and $\inf$ are monotonic with regards to set inclusion, although we have to prove that this is actually the case.

We will often use commonsensical shorthand such as $\varphi+\psi+\theta$ for $\varphi +(\psi+\theta)$, $\varphi - \psi$ for $\varphi + (-1)\cdot \psi$, and $|\varphi|$ for $\max(\varphi,-\varphi)$. We will abbreviate consecutive quantifiers with expressions such as $\sup_{xy}$. By an abuse of notation, we will also write $e(x,y)$ as shorthand for the formula $\inf_{z\sqin y}d(x,z)$.

In the context of a metric set structure $M$, we may often refer to formulas with parameters, such as $\varphi(\xbar,\abar)$ for some $\abar \in M$, as formulas and write them with parameters suppressed. 

Two important properties of formulas in continuous logic are that they only take on values in some bounded interval and that they are always uniformly continuous. The relevant interval and modulus of uniform continuity can be determined by the formula alone. We will need similar facts here.

\begin{lem}\label{lem:formula-v}
  For any metric set structure $(M,d,\sqin)$ and formula $\varphi(\xbar)$,
  \begin{enumerate}
  \item $\varphi^M(\abar) \in [-v(\varphi),v(\varphi)]$ for all $\abar \in M$ and
  \item $\varphi^M : M^{|\xbar|} \to \Rb$ is $2v(\varphi)$-Lipschitz in the sense that for any $\abar,\bbar \in M$, $|\varphi^M(\abar) - \varphi^M(\bbar)| \leq 2v(\varphi)d(\abar,\bbar)$ (where $d(\abar,\bbar)$ is the max metric on tuples).
  \end{enumerate}
\end{lem}
\begin{proof}
  1 follows by an easy induction argument and the fact that $d$ is $[0,1]$-valued. 2 follows similarly from the fact that $(x,y) \mapsto d(x,y)$ is $2$-Lipschitz.
\end{proof}

\subsection{Excision}
\label{sec:excision}

Our comprehension scheme is better defined in terms of its important consequence, rather than directly, as it takes a proof to establish that this property is even axiomatizable. The principle can be informally justified like this:

Suppose that we run a chalk factory and we are contractually obligated to produce pieces of chalk that are no longer than \SI{7.62}{\cm} in length. (The chalk is boxed by another company and needs to fit in their boxes.) Of course, our machine, being cheap, actually produces pieces that are anywhere between roughly \SI{7.4}{\cm} and \SI{7.8}{\cm}. To deal with this, we add a second machine that measures length and rejects pieces that are too long. To maximize our output, we might say that we want it to reject a piece if and only if its length is strictly longer than \SI{7.62}{\cm}, but the realities of physical measurement mean that this is impossible to actually accomplish. Since the penalties for violating the contract are quite harsh, we need to give ourselves some leeway, but we also want to make sure we aren't throwing away too many acceptable pieces of chalk. So we configure the machine to accept chalk if it measures it to be no longer than \SI{7.6}{\cm}. We know that the error of the machine is no more than \SI{0.01}{\cm}, so we can guarantee that we will accept any piece of length at most \SI{7.58}{\cm} and reject any piece of length \SI{7.62}{\cm} or more, but we do not make any promises about the behavior of the machine in the gap between these bounds.\footnote{Reality aside, a similar thing happens in the context of computable analysis: One can write a program that is able to discretely sort computable real numbers in the same manner as our chalk factory, but it is only able to do this if it is allowed to have non-deterministic behavior in some gap. Regardless, this is sufficient for certain purposes.}

This is the manner in which we will approximate comprehension. Given a formula $\varphi(x,\ybar)$ (i.e., a `measurable quantity'), bounds $r < s$, and parameters $\abar$, we promise that we can deliver a set $b$ such that for any $c$, if $\varphi(c,\abar) \leq r$, then $c \sqin b$, and if $\varphi(c,\abar) \geq s$, then $c \not \sqin b$, but we make no commitment about those $c$'s for which $r < \varphi(c,\abar) < s$. %

\begin{defn}\label{defn:exc}
  $(M,d,\sqin)$ satisfies \emph{excision} if for any formula $\varphi(x,\ybar)$, reals $r < s$, and $\abar \in M$, there is a $b \in M$ such that for any $c \in M$, if $\varphi(c,\abar) \leq r$, then $c \sqin b$, and if $c \sqin b$, then $\varphi(c,\abar) < s$.  
\end{defn}

It is straightforward but worthwhile to see how this principle avoid Russell's paradox. We can consider a set $a_r$ satisfying that if $1-e(b,b) \leq r$, then $b \sqin a_r$ and if $b \sqin a_r$, then $1-e(b,b) < 1$. As we pick $r$ closer and closer to $1$, we get better and better approximations of the Russell class, but for each $r$, we consistently have that $r < e(a_r,a_r) < 1$. So we see that while our theory is strictly speaking a $[0,1]$-valued set theory like $\CL_0$, there is something of a qualitative difference in its avoidance of Russell's paradox. While $\CL_0$ is possibly\footnote{Various fragments of this theory were shown to be consistent by a few authors in the 1950s and 60s \cite{Chang1963TheAO,Fenstad-comp,Skolem1957-SKOBZK-3}. A full consistency proof was claimed by White in 1979 \cite{White1979}, but a seemingly fatal gap was discovered by Terui in 2010 \cite{Terui-Error} and consistency remains an open problem.} able to avoid Russell's paradox by Brouwer's fixed point theorem, our theory avoids it by virtue of the required gap between $r$ and $s$ (although these are not unrelated phenomena). %

We are now finally able to define the class of models of our theory directly before defining the theory itself. 

\begin{defn} \label{defn:MSE-models} 
  We say that $(M,d,\sqin)$ is a model of $\MSE$, written $(M,d,\sqin)\models \MSE$ or $M \models \MSE$, if it is a metric set structure that satisfies excision.

  $\MSE$ stands for \emph{Metric Sets with Excision}.
\end{defn}

The following notation will be useful.

\begin{defn}
  Given a metric set structure $M$, a formula $\varphi(x,\ybar)$, tuple $\abar \in M$, and reals $r < s$, we write
  \[
    b = \exc{x}{\varphi(x,\abar)}{r}{s}%
  \]
  to mean that for any $c$, if $\varphi^M(c,\abar) \leq r$, then $c\sqin b$ and if $c \sqin b$, then $\varphi^M(c,\abar) < s$.
\end{defn}
Note of course that $\exc{x}{\varphi(x)}{r}{s}$ is not a uniquely specified object, but if $M$ satisfies excision, it always exists.

It is immediate to show that models of $\MSE$ contain some of the familiar sets one expects to see in a set theory with a universal set.

\begin{prop}
  For any $M \models \MSE$, there are $a,b \in M$ such that for all $c \in M$, $c \not\sqin a$ and $c \sqin b$.
\end{prop}
\begin{proof}
  Let $a = \exc{x}{1}{0}{\tfrac{1}{2}}$ and $b = \exc{x}{0}{\tfrac{1}{2}}{1}$.
\end{proof}

Since these sets are unique by $\HH$-extensionality, we will write $\varnothing^M$ for $\exc{x}{1}{0}{\tfrac{1}{2}}$ and $V^M$ for $\exc{x}{0}{\tfrac{1}{2}}{1}$. We may drop the superscript $M$ if no confusion will arise.

\section{Derived forms of comprehension}

In this section we will show that models of $\MSE$ automatically satisfy certain instances of exact comprehension.

\subsection{Relative excision}
\label{sec:approx-sep}

An important common construction in set theory is \emph{separation}, i.e., comprehension relative to a given set. Ordinarily, separation is an easy consequence of comprehension---$\{x \in A: \varphi(x)\}$ is the same as the set $\{x : x \in A \wedge \varphi(x)\}$---but seeing that excision is merely an approximate form of comprehension, one might worry that we will only be able to find sets that are approximately subsets of other given sets. In other words, if $B$ is a rough approximation of $\{x : x \in A \wedge \varphi(x)\}$, then it would only be the case that $x \sqin B \To e(x,A) < r$ for some small but positive $r$. Fortunately, we are able to build exact subsets of a given set and thereby perform \emph{relative excision}.

\begin{defn}
  For any $a,b \in M$, a metric set structure, we write $a \sqsubseteq b$ to mean that for all $c \in M$, if $c \sqin a$, then $c \sqin b$.
\end{defn}

The following is a special case of relative excision, but we state it first because it is the only form of relative excision we will actually use and it is much easier to prove.

\begin{prop}[Discrete separation]\label{prop:disc-sep}
  Fix $M \models \MSE$ and $\abar$ and $b$ in $M$. For any formula $\varphi(x,\abar)$ and $r<s$, if for all $c \sqin b$, $\varphi^M(c,\abar)\leq r$ or $\varphi^M(c,\abar) \geq s$, then there is an $f \sqsubseteq b$ such that $c \sqin f \Toot \varphi^M(c,\abar) \leq r$.
\end{prop}
\begin{proof}
If $f_n = \exc{x}{\max\left(\varphi(x,\abar)-r,e(x,a) \right)}{0}{2^{-n}}$ for each $n \in \Nb$, then $(f_n)_{n \in \Nb}$ is a Cauchy sequence that limits to the required set. 
\end{proof}

\begin{lem}\label{lem:approx-subset-projection}
  Fix $M \models \MSE$. For any $a,b \in M$ and $\e > 0$, there is a $c \sqsubseteq a$ such that
  \[
    d(b,c) \leq \sup_{f\sqin b}\inf_{g \sqin a}d(f,g) + \e.
  \]
\end{lem}
\begin{proof}
  Fix $\e > 0$. Let $c_0 = b$. For any $n$, let $t_n \coloneqq \sup_{f \sqin c_n}\inf_{g \sqin a}d(f,g)$. At stage $n$, given $c_n$, let
  \[
    c_{n+1} = \exc{x}{\max(e(x,a),e(x,b)-t_n)}{2^{-n-4}\e}{2^{-n-3}\e}.
  \]
 Note that that for any $f \in M$, if $f \sqin a$ and $e(f,c_n) < t_n + 2^{-n-4}\e$, then $f \sqin c_{n+1}$, and if $f \sqin c_{n+1}$, then $e(f,a) < 2^{-n-4}\e$ and $e(f,c_n) < t_n + 2^{-n-3}\e$. In particular, $t_{n+1} \leq 2^{-n-3}\e$. Note also that by the definition of $t_n$, we have that for any $g \sqin c_n$, there is an $f \sqin a$ such that $d(f,g) < t_n + 2^{-n-4}\e$. Such a $g$ must be an element of $c_{n+1}$. Since we can do this for any $f \sqin c_n$, we have that $d(c_n,c_{n+1}) \leq \max(t_n+2^{-n-4}\e,t_n + 2^{-n-3}\e) = t_n + 2^{-n-3}\e$. Hence, for any $n > 0$, we have that $d(c_n,c_{n+1}) \leq 2^{-n-2}\e + 2^{-n-3}\e < 2^{-n-1}\e$. Therefore $(c_n)_{n\in \Nb}$ is a Cauchy sequence. Let $c = \lim_{n \to \infty}c_n$.

  Since $t_n \to 0$ as $n \to \infty$, we have that $c \sqsubseteq a$. Now we just need to verify that $d(b,c) \leq t_0 + \e$. Our estimates give that
  \[
d(b,c) \leq d(c_0,c_1) + d(c_1,c) \leq t_0 + 2^{-4}\e + \sum_{n=1}^\infty 2^{-n-1}\e < t_0 + \e,
\]
as required.
\end{proof}

\begin{prop}[Relative excision]\label{prop:sep}
  If $M\models \MSE$, then for any real-valued formula $\varphi(x,\ybar)$, $\abar,b \in M$, and reals $r < s$, there is a $c \in M$ with $c \sqsubseteq b$ such that for any $f \sqin b$, if $\varphi(f,\abar) \leq r$, then $f \sqin c$, and if $f \sqin c$, then $\varphi(c,\abar) < s$.
\end{prop}
\begin{proof}
  Fix $\varphi(x,\ybar) \in \Lsq$, $\abar$ and $b$ in $M$, and $\delta > 0$ with $\delta < \frac{1}{4v(\varphi)}$.

  Let $c = \exc{x}{\max(\varphi(x,\abar),e(x,b))}{0}{\delta}$. Note that $\sup_{x\sqin c}\inf_{y\sqin b}d(x,y) \leq \frac{1}{2}\delta$. Apply \cref{lem:approx-subset-projection} to $c$ to get an $f \sqsubseteq b$ such that $d(c,f) < \delta$. We now have that for any $g$, if $g \sqin b$ and $\varphi(g,\abar) \leq 0$, then $e(g,f) < \delta$ and if $e(g,f)\leq \frac{1}{6v(\varphi)}$, then $e(g,c) \leq \frac{1}{4v(\varphi)} + \delta < \frac{1}{2v(\varphi)}$ and so $\varphi(g,\abar) < 2v(\varphi)\frac{1}{2v(\varphi)}$.

  Since we can do this for any $\varphi(x,\ybar)$ and $\abar \in M$, we have that $M$ satisfies \cref{defn:ax-sch-exc} relative to the set $b$. The proposition then follows by repeating the proofs of \cref{lem:fat-comp} and \cref{prop:MSE-CL-char} relative to the set $b$.
\end{proof}

In light of Propositions~\ref{prop:disc-sep} and \ref{prop:sep}, we will write
\[
c = \exc{x \sqin b}{\varphi(x,\abar)}{r}{s}
\]
to mean that for any $f \in M$, if $f \sqin b$ and $\varphi^M(f,\abar) \leq r$, then $f \sqin c$ and if $f \sqin c$, then $f \sqin b$ and $\varphi^M(f,\abar) < s$.

\subsection{Comprehension for definable classes}
\label{sec:defbl-classes}

Given \cref{prop:sep}, one might be tempted to ask whether we can just outright show that models of $\MSE$ satisfy a more conventional form of comprehension. Suppose we have a formula $\varphi(x)$ and we wish to form the set $\{x \in M : \varphi^M(x) = 0\}$. Could we not just form the sequence $a_n = \exc{x}{|\varphi(x)|}{0}{2^{-n}}$ and take the limit? While we are perfectly able to form this sequence externally, the difficulty is that it will in general fail to be Cauchy. 

Regardless, there are times when such a sequence of approximations does actually converge in the Hausdorff metric, giving us an instance of exact comprehension. This happens precisely when $\{x \in M: \varphi^M(x) = 0\}$ is a \emph{definable set} in the sense of continuous logic, although in the context of a set theory it would be more appropriate to refer to these as \emph{definable classes}. For the sake of this paper, we will not need the full generality of definable sets. %

\begin{defn}\label{defn:defnbl-class}
  A closed subset $D \subseteq M^n$ is a \emph{definable class} if the function $\xbar \mapsto \inf_{\abar \in D}d(\xbar,\abar)$ is a uniformly convergent limit of functions of the form $\varphi^M(\xbar,\bbar)$ for $\varphi(\xbar,\ybar) \in \Lsq$ and $\bbar \in M$. $D$ is definable \emph{without parameters} if its definability is witnessed by formulas without parameters.
  
  $D$ is an \emph{explicitly definable class}\footnote{There is no standard term for explicit definability in continuous logic, as it's not a wholly natural concept.} if there is a $\varphi(\xbar,\ybar) \in \Lsq$ and a tuple $\bbar$ such that $\inf_{\abar \in D}d(\xbar,\abar) = \varphi^M(\xbar,\bbar)$.

  $\xbar \mapsto \inf_{\abar \in D}d(\xbar,\abar)$ is called the \emph{distance predicate of $D$}, which we may also write as $e(\xbar,D)$.
\end{defn}

This definition is perhaps most strongly motivated by the fact that definable classes are precisely those that admit relative quantification. %

\begin{lem}\label{lem:rel-quan}
  For any metric set structure $M$ and $\abar \in M$, if $\varphi(\xbar,\abar)$ is the distance predicate of a definable class $D \subseteq M^n$, then for any $\psi(\xbar,\ybar,\zbar)$ and $\cbar \in M$,
  \[
    \left(\inf_{\xbar}\min(\psi(\xbar,\cbar,\bbar) + 2v(\psi) \varphi(\xbar,\abar),v(\psi))\right)^M = \inf\{\psi^M(\fbar,\cbar,\bbar) : \fbar \in D\},
  \]
  where $\inf \varnothing$ is understood to be $v(\psi)$. 
\end{lem}
\begin{proof}
  Let $r = \inf_{\xbar}\min(\psi(\xbar,\cbar,\bbar) + 2v(\psi) \varphi(\xbar,\abar),v(\psi))$ and $s = \inf\{\psi^M(\fbar,\cbar,\bbar) : \fbar \in D\}$.

  If $D$ is empty, then $\varphi(\xbar,\abar) = 1$ and the result holds.

  If $D$ is not empty, then we clearly have that $r \leq s$
  since $\psi(\xbar,\cbar,\bbar) \in [-v(\psi),v(\psi)]$ by \cref{lem:formula-v}.
  For the other direction, fix $\gbar \in M$. For any $\e > 0$, there is an $\fbar \in D$ such that $d(\fbar,\gbar) < e(\gbar,D)+ \e$. By \cref{lem:formula-v}, $\xbar\mapsto \psi(\xbar,\cbar,\bbar)$ is $2v(\psi)$-Lipschitz, so $\psi(\gbar,\cbar,\bbar)+2v(\psi)\varphi(\gbar,\abar) \geq \psi(\fbar,\cbar,\bbar)$ and therefore $\min(\psi(\gbar,\cbar,\bbar)+2v(\psi)\varphi(\gbar,\abar),v(\psi)) \geq \psi(\fbar,\cbar,\bbar)$. Since we can do this for any $\gbar \in M$, we have that $r \geq s$ and we are done.
\end{proof}

In continuous logic generally, definable classes can be characterized as those sets that admit relative quantification in the same sense as \cref{lem:rel-quan}. In models of $\MSE$ moreover, definable classes of $1$-tuples correspond precisely to sets. In particular, every definable class is explicitly definable by $e(x,a)$ for some $a$.

\begin{prop}\label{prop:def-class-char}
  Let $M \models \MSE$. A closed set $D \subseteq M$ is a definable class if and only if there is an $a \in M$ such that $D = \{b \in M : b \sqin a\}$.
\end{prop}
\begin{proof}
  The $\Leftarrow$ direction is obvious. To show the $\To$ direction, find a formula $\varphi_n(x,\abar_n)$  for each $n \in \Nb$ such that $\sup_x|\varphi_n(x,\abar_n)-e(x,D)| < 2^{-n}$. Let $b_n = \exc{x}{\varphi_n(x,\abar_n)}{2^{-n}}{2^{-n+1}}$. Note that if $c \in D$, then $c \sqin b_n$ and if $c \sqin b_n$, then $e(c,D) < 2^{-n+1}+2^{-n} < 2^{-n+1}$. This implies that $d_{\HH}(\{x \in M : x \sqin b_n\},D) \leq 2^{-n+1}$, and so the sequence $(b_n)_{n \in \Nb}$ is a Cauchy sequence and $b = \lim_{n \to \infty}b_n$ has the property that $c \sqin b$ if and only if $c \in D$.
\end{proof}

\cref{prop:def-class-char} allows us to answer some very basic questions that we haven't resolved yet.

\begin{cor}\label{cor:basic-stuff}
  Fix $M \models \MSE$.
  \begin{enumerate}
  \item \emph{(Singletons)} For any $a \in M$, there is a $b \in M$ such that $c \sqin b$ if and only if $c = a$.
  \item \emph{(Finite unions)} For any $a,b \in M$, there is a $c \in M$ such that $f \sqin c$ if and only if $f \sqin a$ or $f \sqin b$.
  \item \emph{(Finite sets)} For any $a_0,\dots,a_{n-1} \in M$, there is a $b \in M$ such that $c \sqin b$ if and only if $c=a_i$ for some $i<n$.
  \item \emph{(Closure-of-unions)} For any $a \in M$, there is a $b \in M$ such that $c \sqin b$ if and only if $b$ is in the metric closure of $\{f \in M : (\exists g \sqin a)f \sqin g\}$.
  \end{enumerate}
\end{cor}
\begin{proof}
  \begin{enumerate}
  \item This is witnessed by the formula $d(x,a)$.
  \item This is witnessed by the formula $\min(e(x,a),e(x,b))$.
  \item This follows from 1 and 2 by induction.
  \item This is witnessed by the formula $\inf_{y \sqin a}e(x,y)$. \qedhere
  \end{enumerate}
\end{proof}

Unfortunately, however, it is generally not the case that the distance to the intersection of two sets $X$ and $Y$ can be computed from the distances to $X$ and $Y$. As such, \emph{we cannot establish the existence of intersections in general.} We do, however, get a kind of approximate intersection in the form of $\exc{x\sqin a}{e(x,b)}{0}{\e}$.

 A minor corollary of \cref{cor:basic-stuff} is that models of $\MSE$ satisfy $d(x,y) = \sup_{z}|e(x,z)-e(y,z)|$, as witnessed by the singleton $\{x\}$. In $[0,1]$-valued set theories, the quantity $1-\sup_z|e(x,z)-e(y,z)|$ is often referred to a \emph{Leibniz equality}, as it represents the degree to which $x$ and $y$ cannot be discerned from each other.

In light of \cref{cor:basic-stuff}, we will write $\{a_0,a_1,\dots,a_{n-1}\}$ for the finite set containing $a_0,a_1,\dots,a_{n-1}$, $a \sqcup b$ for the union of $a$ and $b$, and $\ol{\bigsqcup a}$ for the closure of the union of the elements of $a$.

  Now that we have the ability to form finite sets by \cref{cor:basic-stuff}, we are free to code ordered pairs. While we certainly could use the standard Kuratowski ordered pair, Wiener's earlier definition is actually preferable to us for technical reasons.\footnote{If $\langle a,b \rangle \coloneqq \{\{a\},\{a,b\}\}$, then a straightforward but tedious calculation shows that $d(\langle a,b \rangle,\langle c,f \rangle) = \max(d(a,c),\min(d(b,f),\max(d(a,f),d(b,c))))$ and so $d(\langle a,b \rangle,\langle c,f \rangle) \leq d(ab,cf) \leq 3d(\langle a,b \rangle,\langle c,f \rangle)$. Setting $\langle a,b \rangle=\langle 2,0 \rangle$ and $\langle c,f \rangle = \langle 1,3 \rangle$ in $\Rb$ shows that this is sharp. If $d$ is an ultrametric however, we do get $d(\langle a,b \rangle,\langle c,f \rangle) = d(ab,cf)$.} As such, we will write $\langle a,b \rangle$ for $\{\{\{a\},\varnothing\},\{\{b\}\}\}$. Recall that $d(ab,cf)\coloneqq \max(d(a,c),d(b,f))$.
\begin{lem}\label{lem:ord-pair-max-metric}
  Let $M\models \MSE$. For any $a,b,c,f \in M$,
  \[
d(\langle a,b \rangle,\langle c,d \rangle) = d(ab,cf).
\]
\end{lem}
\begin{proof}
  Let $A = \{\{a\},\varnothing\}$, $B = \{\{b\}\}$, $C = \{\{c\},\varnothing\}$, and $F = \{\{f\}\}$. We have that
  \[
    d(\{A,B\},\{C,F\}) = \max(e(A,\{C,F\}), e(B,\{C,F\}),e(C,\{A,B\}),e(F,\{A,B\})).
  \]
  For any $x$, $d(\{x\},\varnothing) = 1$. This implies that $d(\{\{x\},\varnothing\},\{\{y\}\}) = 1$ for any $x$ and $y$ as well. This implies, for instance, that
  \[
    e(A,\{C,F\}) = \min(d(A,C),d(A,F)) = \min(d(A,C),1) =  d(A,C). 
  \]
  This together with similar facts for the other three terms implies that
  \begin{align*}
    d(\{A,B\},\{C,F\}) &= \max(d(A,C),d(B,F),d(C,A),d(F,B)) \\
    &= \max(d(A,C),d(B,F)).
\end{align*}
Finally, $d(\{\{x\}\},\{\{y\}\}) = d(x,y)$ and $d(\{\{x\},\varnothing\},\{\{y\},\varnothing\}) = d(x,y)$ for any $x$ and $y$, so we have that $d(\langle a,b \rangle,\langle c,f \rangle) = d(\{A,B\},\{C,F\}) = d(ab,cf)$, as required.
\end{proof}

\cref{lem:ord-pair-max-metric} means that a sequence of ordered pairs can only converge to an ordered pair and that convergence of sequences of ordered pairs behaves in the expected way. In particular, the class of ordered pairs is closed.

  With a little more work we can establish the existence of power sets.

  \begin{prop}[Power sets]
    For any $a \in M$, there is a $b \in M$ such that $c \sqin b$ if and only if $c \sqsubseteq a$. %
  \end{prop}
  \begin{proof}
    First note that by basic properties of the Hausdorff metric, the class $\Pc(a) \coloneqq \{c \in M : c \sqsubseteq a\}$ is necessarily closed. By \cref{lem:approx-subset-projection}, we know that for any $f \in M$,
    \[
      \inf_{g \in \Pc(a)}d(f,g)\leq \sup_{y \sqin f}\inf_{z \sqin a}d(y,z).
    \]
    On the other hand, for any $g \in \Pc(a)$, we must have that
    \[
      \sup_{y \sqin f}\inf_{z \sqin a}d(y,z)\leq \sup_{y \sqin f}\inf_{z \sqin g}d(y,z)\leq d(f,g),
    \]
    by monotonicity of $\inf$. Therefore,
    \[
      \sup_{y \sqin f}\inf_{z \sqin a}d(y,z) \leq \inf_{g \in \Pc(a)}d(f,g),
    \]
    and so the two quantities are actually equal. Hence $\Pc(a)$ is a definable class and is coextensive with some element $b\in M$ by \cref{prop:def-class-char}.
  \end{proof}

  We will write $\Pc(a)$ for the power set of $a$.

\subsection{Definable functions and replacement}
\label{sec:definable-functions}

Now that we are confident that ordered pairs exist, the next natural thing to consider is Cartesian products. In order to show that the class $a \times b \coloneqq \{\langle c,f \rangle: c \sqin a,~f\sqin b\}$ is definable and therefore a set, what we would like to be able to do is write a real-valued formula like this:
\[
\varphi(x) = \inf_{y \sqin a}\inf_{z \sqin b}d(x,\{\{\{a\},\varnothing\},\{\{b\}\}\}).
\]
If we did have this formula, $\varphi(x)$ would of course be the point-set distance from $x$ to the class $a \times b$. The issue is that the functions $x \mapsto \{x\}$ and $(x,y) \mapsto \{x,y\}$ and the constant $\varnothing$ are not formally part of our logic.

Practically, however,  it is commonly understood in the context of discrete logic that it is safe to pretend that certain functions---namely the \emph{definable} functions---are formally part of the language in the following sense: Given a discrete structure $M$, a function $f:M^n \to M$ is \emph{definable} if and only if for every formula $\varphi(\xbar,y,\zbar)$, there is a formula $\psi(\xbar,\zbar)$ such that for any $\abar,\bbar \in M$, $M \models \varphi(\abar,f(\abar),\bbar)$ if and only if $M \models \psi(\abar,\bbar)$. It is easy to show that this is equivalent to the graph of $f$ being a definable subset of $M^{n+1}$.

In continuous logic, a similar thing can be done:

\begin{defn}
  Given a set $X \subseteq M^n$, a function $f : X \to M$ is \emph{definable} if for every $\e > 0$, there is a $\varphi(\xbar,y,\zbar) \in \Lsq$ and a $\cbar \in M$ such that for any $\abar \in X$ and $b \in M$, $|d(f(\abar),b) - \varphi^M(\abar,b,\cbar)| < \e$.

  $f$ is \emph{explicitly definable} if there is a $\varphi(\xbar,y,\zbar) \in \Lsq$ and a tuple $\cbar$ such that $d(f(\abar),b) = \varphi^M(\abar,b,\cbar)$ for every $\abar \in X$ and $b \in M^{n+1}$.

  If $X = M^n$, we say that $f$ is an \emph{(explicitly) definable total function}. Otherwise it is an \emph{(explicitly) definable partial function}.\footnote{Beware that, unlike in discrete first-order logic, definable partial functions do not always extend to definable total functions in continuous logic \cite[C.1.2]{HansonThesis}.}
    
  Given a function $g : M^n \to M$, we say that $g$ is \emph{(explicitly) definable on $X$} if $g \res X$ is (explicitly) definable.
\end{defn}

When $X$ is itself definable, it is not too hard to show that $f$ is definable if and only if it is uniformly continuous and its graph is definable (in the sense of \cref{defn:defnbl-class} relative to the max metric on tuples). 

Again, while more general statements can be made (see \cite[Sec.~9]{MTFMS}), we really only need explicitly definable functions.\footnote{Note though that with definable sets, in the special context of models of $\MSE$, all definable sets are ultimately explicitly definable. It seems unlikely that this will be true for non-Lipschitz definable functions.}

\begin{lem}\label{prop:defbl-fun-sub}
  For any metric set structure $M$, formula $\varphi(\xbar,y,\zbar)$, and explicitly definable function $f(\xbar)$ with domain $X \subseteq M^n$, there is a formula $\psi(\xbar,\zbar)$ (possibly with parameters) such that for all $\abar \in X$ and $\bbar \in M$, $\psi^M(\abar,\bbar) = \varphi^M(\abar,f(\abar),\bbar)$.
\end{lem}
\begin{proof}
  Let $f(\xbar)$ be defined on $X$ by $\chi(\xbar,y)$ (possibly with parameters). By the same argument as in the proof of \cref{lem:rel-quan}, the formula $\inf_y\varphi(\xbar,y,\zbar) + 2v(\varphi) \chi(\xbar,y)$ is the required $\psi(\xbar,\zbar)$.
\end{proof}

A corollary of \cref{prop:defbl-fun-sub} is that compositions of explicitly definable functions are explicitly definable. (This is also true of definable functions, but we will not need it.)

It is fairly immediate that the operations we established in \cref{sec:defbl-classes} are in fact definable:

\begin{prop}\label{prop:defbl-omnibus}
  Let $M \models \MSE$. The following functions are explicitly definable.
  \begin{enumerate}
  \item $()\mapsto \varnothing^M$.
  \item $()\mapsto V^M$.
  \item $(x_0,x_1,\dots,x_{n-1}) \mapsto \{x_0,x_1,\dots,x_{n-1}\}$.
  \item $(x,y) \mapsto x\sqcup y$.
  \item $x \mapsto \overline{\bigsqcup x}$.
  \item\label{ordered-pair} $(x,y) \mapsto \langle x,y \rangle$.
  \item $x \mapsto \Pc(x)$.
  \end{enumerate}
\end{prop}
\begin{proof} The definability of these functions are witnessed by the following formulas.
  \begin{enumerate}
  \item $d(y,\varnothing) = 1-\inf_{z \sqin y}d(z,z)$.
  \item $d(y,V) = \sup_{z}e(z,y)$.
  \item $d(y,\{x_0,x_1,\dots,x_{n-1}\}) = \sup_{z}|e(z,y) - \min(d(z,x_0),\dots,d(z,x_{n-1}))|$.
  \item $d(z,x\sqcup y) = \sup_{w}|e(w,z) - \min(e(w,x),e(w,y))|$.
  \item $d\left( y, \overline{\bigsqcup x} \right) = \sup_z\left|e(z,y) - \inf_{w \sqin x} e(z,w)\right|$.
    \addtocounter{enumi}{1}
  \item $d(y,\Pc(x)) = \sup_z\left| e(z,y) - \sup_{u\sqin z}\inf_{v\sqin x}d(u,v)\right|$. %
  \end{enumerate}
  \ref{ordered-pair} follows from the fact that $(x,y) \mapsto \langle x,y \rangle$ is a composition of other explicitly definable functions.
\end{proof}

Now finally we can return to the question of forming a Cartesian product of two sets. The relevant fact is this:

\begin{prop}[Images of definable functions]\label{prop:def-fun-im}
  Let $M \models \MSE$. For any definable function $f: X \to M$ and any $a_0,\dots,a_{n-1} \in M$, if $\bbar \in X$ for any tuple $\bbar$ satisfying $b_i\sqin a_i$ for each $i<n$, then the metric closure of $\{f(b_0,\dots,b_{n-1}) : b_0 \sqin a_0,\dots,b_{n-1}\sqin a_{n-1}\}$ is a definable class.
\end{prop}
\begin{proof}
  $\inf_{x_0 \sqin a_0}\dots \inf_{x_{n-1} \sqin a_{n-1}}d(x,f(x_0,\dots,x_{n-1}))$ is clearly the distance predicate of the class in question. By \cref{prop:defbl-fun-sub}, this is equivalent to a formula.
\end{proof}

\begin{cor}[Cartesian products]\label{cor:Cart-prod}
  Let $M \models \MSE$. For any $a,b \in M$, there is a $c \in M$ such that $f \sqin c$ if and only if $f = \langle g,h \rangle$ for some $g \sqin a$ and $h \sqin b$.
\end{cor}
\begin{proof}
By \cref{prop:def-fun-im}, the metric closure of $a \times b \coloneqq \{\langle x,y \rangle : x \sqin a,~y\sqin b\}$ is a definable class. By the discussion at the end of \cref{sec:defbl-classes}, $a \times b$ is already metrically closed, so we have that it is a definable class. By \cref{prop:def-class-char}, we have that the required $c$ exists. 
\end{proof}

We will write $a \times b$ for the set whose existence is established in \cref{cor:Cart-prod}.

One thing to note is that the proof of \cref{cor:Cart-prod} actually establishes that the function $(x,y) \mapsto x \times y$ is explicitly definable as witnessed by the formula
\(
 \inf_{w \sqin x}\inf_{u \sqin y}d(z,\langle w,u \rangle).
\)

It is occasionally useful to be able to project sets of ordered pairs onto their coordinates. Since the projection function $\pi_0(\langle x,y \rangle) = x$ is only partially defined, this is the first time we need the added generality of being able to talk about definable partial functions.

\begin{prop}\label{prop:proj-is-defbl}
  Let $M \models \MSE$. Let $\pi_0$ and $\pi_1$ be the functions on the class of ordered pairs defined by $\pi_0(\langle x,y \rangle) =x$ and $\pi_1(\langle x,y \rangle) = y$. $\pi_0$ and $\pi_1$ are explicitly definable.
\end{prop}
\begin{proof}
  This is witnessed by the formulas $\inf_zd(\langle y,z \rangle,x)$ and $\inf_zd(\langle z,y \rangle,x)$. 
\end{proof}

The following facts will also be useful.

\begin{lem}\label{lem:uniformly-definable-is-definable-function}
  Fix closed sets $X \subseteq M^n$ and $Y\subseteq M^{n+1}$. Suppose that there is a formula $\varphi(\xbar,y)$ such that for every $\abar \in X$ and $b \in M$, $\varphi^M(\abar,b) = e(b,\{y : (\abar,y) \in Y\})$. Then there is an explicitly definable function $f: X \to M$ such that for every $\abar \in X$, $f(\abar)$ is coextensive with $\{y : (\abar,y) \in Y\}$.
\end{lem}
\begin{proof}
  The formula $\psi(\xbar,z) = \inf_{y}|e(y,z)-\varphi(\xbar,y)|$ witnesses that the required function is explicitly definable.
\end{proof}

\begin{lem}\label{lem:im-of-def-is-def}
  If $f: a \to M$ is an explicitly definable partial function on some set $a$, then the map $x \mapsto \overline{\{f(y) : y \sqin x\}}$ is an explicitly definable partial function on the set $\Pc(a)$.
\end{lem}
\begin{proof}
  Let $\varphi(y,z)$ be a formula (possibly with parameters) such that $\varphi^M(b,c) = d(b,f(c))$ for all $b\sqin a$ and $c \in M$. We now have that for any $g\sqsubseteq a$, $\inf_{z \sqin g}\varphi(y,z)$ is the distance predicate of $\ol{\{f(y):y \sqin g\}}$. Therefore the required function is definable by \cref{lem:uniformly-definable-is-definable-function}.
\end{proof}

\subsection{Quotients by discrete equivalence relations}
\label{sec:disc-eq-reln}

A common technique is passing from an equivalence relation to its set of equivalence classes. We are able to do this for discrete equivalence relations.

\begin{defn}\label{defn:disc-eq-reln}
Fix a metric set structure $M$. Given a set $a \in M$, a formula $\varphi(x,y)$ is a \emph{discrete equivalence relation on $a$} if for all $b,c \sqin a$, $\varphi^M(b,c)$ is either $0$ or $1$ and $\varphi^M(x,y) = 1$ is an equivalence relation on $\{b \in M : b \sqin a\}$.
\end{defn}

First we need a small observation.

\begin{lem}\label{lem:eq-class-dist-pred}
  Let $M \models \MSE$. If $\varphi(x,y)$ is a discrete equivalence relation on $a \in M$, then for any $b \sqin a$,
  \[
\psi(x,b) \coloneqq \inf_{y \sqin a}\max(\varphi(y,b),d(x,y))
\]
is the distance predicate of the $\varphi$-equivalence class of $b$.\hfill $\qed$ 
\end{lem}

\begin{prop}\label{prop:disc-eq-reln-quot}
  Let $M \models \MSE$. If $\varphi(x,y)$ is a discrete equivalence relation on $a \in M$, then there is a $b \in M$ containing precisely the $\varphi$-equivalence classes of $a$.

  Moreover, if $c$ is a set such that $\varphi(x,y)$ is a discrete equivalence relation on every $a \sqin c$, then the map taking $a$ to the set of $\varphi$-equivalence classes of $a$ is an explicitly definable partial function on $c$.
\end{prop}
\begin{proof}
  The formula in \cref{lem:eq-class-dist-pred} defines a partial function on $a$ that maps elements to their $\varphi$-equivalence classes. Therefore, by \cref{prop:def-fun-im}, the closure of the class of $\varphi$-equivalence classes of $a$ is a set in $M$. By \cref{lem:eq-class-dist-pred}, the class of $\varphi$-equivalence classes is closed, so we are done.

  The `Moreover' statement follows from \cref{lem:uniformly-definable-is-definable-function}.
\end{proof}

\section{Uniformly discrete sets and ordinary mathematics}
\label{sec:disc-set}

A common feature of models of set theories that implement some kind of nearly unrestricted comprehension is that they have a class of tame sets in which unrestricted separation is consistent. In $\NFU$, the strongly cantorian sets are well behaved in this way, and in $\GPK^+_\infty$, the closed sets of isolated points are likewise well behaved. In the context of $\MSE$, the analogously well-behaved class seems to be that of the \emph{uniformly discrete} sets.

\begin{defn}
  In a metric set structure $M$, an element $a \in M$ is \emph{$\e$-discrete} if for any $b,c \sqin a$, either $b=c$ or $d(b,c) \geq \e$. $a$ is \emph{uniformly discrete} if it is $\e$-discrete for some $\e > 0$.
\end{defn}

\begin{defn}
  For any $a$ and $b$, the \emph{disjoint union of $a$ and $b$}, written $a \boxplus b$, is $(a\times \{\varnothing\})\sqcup(b\times \{\{\varnothing\}\})$.
\end{defn}

Note that it is immediate that $(x,y) \mapsto x\boxplus y$ is an explicitly definable function.

\begin{lem}\label{lem:unif-disc-prod-power}
  Let $M \models \MSE$. If $a,b \in M$ are $\e$-discrete, then $a \boxplus b$, $a \times b$, and $\Pc(a)$ are $\e$-discrete.
\end{lem}
\begin{proof}
  If $a$ is empty, then the statements in the lemma are trivial, and if $b$ is empty, then the statements for $a \boxplus b$ and $a \times b$ are trivial, so assume that $a$ and $b$ are both non-empty. Fix $c,f \sqin a$ and $g,f \sqin b$.

  For the disjoint union, we have that $d(\langle c,\varnothing \rangle,\langle g,\{\varnothing\} \rangle) = d(\varnothing,\{\varnothing\}) = 1$. Furthermore, $d(\langle c,\varnothing \rangle,\langle f,\varnothing \rangle) = d(c,f)$ and $d(\langle g,\{\varnothing\} \rangle,\langle f,\{\varnothing\} \rangle) = d(g,f)$, so $a \boxplus b$ is $\e$-discrete.
  
  For the Cartesian product, if $\langle c,g \rangle \neq \langle f,h \rangle$, then either $c\neq f$ or $g \neq h$. In either case we have that $d(\langle c,g \rangle,\langle f,h \rangle) \geq \e$.

  For the power set, if $c \neq f$, then we may assume without loss of generality that there is a $g$ such that $g \sqin c$ but $g \not \sqin f$. Since $g \neq h$ for all $h \sqin f$, we have that $e(g,h) \geq \e$ (since $a$ is $\e$-discrete). Therefore $d(c,f) \geq \e$.
\end{proof}

Note that \cref{lem:unif-disc-prod-power} relies on our use of Wiener pairs over Kuratowski pairs. While seemingly a cosmetic issue, this will eventually matter in \cref{lem:unif-disc-prod-power}.

Within uniformly discrete sets, we are generally able to reason in a familiar discrete manner. To distinguish discrete formulas from real-valued formulas, we will usually write discrete formulas with capital Greek letters.

\begin{defn}
Given a tuple $\abar$ of formal type variables, we write $\abar^\ast$ for the smallest collection of expressions containing $\{a_0,a_1,\dots\}$ and containing (the formal expressions) $b\times c$ and $\Pc(b)$ for any $b,c \in A^\ast$.  
  
  We write $\Lc_{\mathrm{dis}}(\abar)$ for the smallest collection of formulas satisfying the following (where $b$ and $c$ are elements of $\abar^\ast$):
  \begin{itemize}
  \item For any $x{:}b$ and $y{:}c$, $x=y$ is in $\Lc_{\mathrm{dis}}(\abar)$.
  \item For any $x{:}b$, $y{:}c$, and $z{:}b\times c$, $\langle x,y \rangle = z$ is in $\Lc_{\mathrm{dis}}(\abar)$.
  \item For any $x{:}b$ and $y{:}\Pc(b)$, $x \sqin y$ is in $\Lc_{\mathrm{dis}}(\abar)$.
  \item For $\Phi,\Psi \in \Lc_{\mathrm{dis}}(\abar)$, $\Phi\wedge \Psi$, $\Phi\vee\Psi$, $\Phi \to \Psi$, and $\neg \Phi$ are in $\Lc_{\mathrm{dis}}(\abar)$.
  \item For $\Phi \in \Lc_{\mathrm{dis}}(\abar)$ and $x{:}b$, $(\exists x{:}b)\Phi$ and $(\forall x{:}b)\Phi$ are in $\Lc_{\mathrm{dis}}(\abar)$.
  \end{itemize}
If we wish to specify the variables and formal type variables of a formula $\Phi \in \Lc_{\mathrm{dis}}(\abar)$, we will write $\Phi(\xbar;\abar)$ (where the free variables of $\Phi$ are among $\xbar$).
\end{defn}

The formal type variables in a formula of $\Lc_{\mathrm{dis}}(\abar)$ are intended to be interpreted as uniformly discrete sets in a model of $\MSE$, in which case a variable of type $a$ is allowed to take on values in $a$. The interpretation of a formula is then clear:

\begin{defn}
  Fix $M \models \MSE$. Given a formula $\Phi(\xbar;\abar) \in \Lc_{\mathrm{dis}}(\abar)$, a tuple $\bbar \in M$ of the same length as $\abar$, and a tuple $\cbar$ of the same length as $\xbar$, we say that $\Phi(\cbar;\bbar)$ is \emph{well-typed} if for each $c_i \in \cbar$, if $x_i$ is a variable of type $t(\abar)$ (where $t(\abar)$ is a formal type expression), then $c_i \sqin t(\bbar)$ (where $t$ is now interpreted as a literal expression involving the functions $\times$ and $\Pc$ in $M$).

  If $\Phi(\cbar;\bbar)$ is well-typed, we write $M \models \Phi(\cbar;\bbar)$ to mean that $(M,\sqin)$ satisfies $\Phi(\cbar;\bbar)$ as a discrete structure (where quantifiers such as $\exists x {:}t(\bbar)$ are interpreted as $\exists x \sqin t(\bbar)$).
\end{defn}

Now, we will see that as long as the sets in $\bbar$ are $\e$-discrete, we can express these kinds of discrete formulas as real-valued formulas in a mostly uniform way.

\begin{defn}
  Given a tuple $\abar$ of formal type variables, a formula $\Phi(\xbar) \in \Lc_{\mathrm{dis}}(\abar)$, and $\e > 0$, we write $\cset{\Phi}_{\e}(\xbar;\abar)$ for the real-valued formula defined by the following inductive procedure:
  \begin{itemize}
  \item $\cset{x=y}_\e=\max\left( 1-\frac{1}{\e} d(x,y),0 \right)$.
  \item $\cset{\langle x,y \rangle = z}_\e = \max\left( 1-\frac{1}{\e}d(\langle x,y \rangle,z),0 \right)$.
  \item $\cset{x \sqin y}_\e = \max\left( 1-\frac{1}{\e}e(x,y),0 \right)$.
  \item $\cset{\Phi\wedge\Psi}_\e = \min\left(\cset{\Phi}_\e,\cset{\Psi}_\e\right)$.
  \item $\cset{\neg \Phi}_{\e} = 1-\cset{\Phi}_{\e}$.
  \item $\cset{(\exists x{:}a) \Phi}_{\e} = \sup_{x \sqin a}\cset{\Phi}_{\e}$.
  \end{itemize}
  The other Boolean connectives and the universal quantifier are defined from the above in the typical way.
\end{defn}

\begin{prop}\label{prop:disc-lang-unif}
  Fix $M\models \MSE$. For any formal tuple $\abar$ of type variable and any formula $\Phi(\xbar;\abar) \in \Lc_{\mathrm{dis}}(\abar)$, we have for any $\e$-discrete sets $\bbar \in M$ and any $\cbar \in M$ such that $\Phi(\bbar;\cbar)$ is well-typed, $M \models \Phi(\bbar;\cbar)$ if and only if $(\cset{\Phi}_{\e}(\bbar;\cbar))^M = 1$ and $M \models \neg \Phi(\bbar;\cbar)$ if and only if $(\cset{\Phi}_{\e}(\bbar;\cbar))^M = 0$.
\end{prop}
\begin{proof}
 This follows immediately by induction on the construction of formulas in $\Lc_{\mathrm{dis}}(\abar)$.
\end{proof}

Given \cref{prop:disc-lang-unif}, we can now confidently talk about familiar discrete concepts in the context of uniformly discrete sets. In particular, we can develop the notion of cardinalities.

\begin{defn}
  Given a uniformly discrete sets $a,b \in M \models \MSE$, we write $M \models a \approx b$ to mean that there is a $c \sqsubseteq a \times b$ that is the graph of a bijection between $a$ and $b$.
\end{defn}

It is clear that there is an $\Lc_{\dis}(a,b)$-formula $\eta(x,y)$ with the property that if $a$ and $b$ are $\e$-discrete, $c \sqsubseteq a$, and $f \sqsubseteq b$, then $M \models \cset{\eta(c,f)}_\e$ if and only if $M \models c \approx f$. We will write $x \approx_{a,b}y$ for this formula. We will write $\approx_a$ for $\approx_{a,a}$. It is immediate that $\cset{x \approx_{a} y}_\e$ is an equivalence relation on $\Pc(a)$ whenever $a$ is $\e$-discrete.

\begin{defn}
  Given an $\e$-discrete set $a$, a \emph{cardinal of $a$} is a $\cset{x\approx_ay}_\e$-equivalence class. We write $\Card_a$ for the collection of cardinals of $a$.

  Given $b \sqsubseteq a$, we write $|b|_a$ for the $\cset{x\approx_a y}_\e$-equivalence class of $b$.
\end{defn}

It follows immediately from the above discussion and \cref{prop:disc-eq-reln-quot} that $\Card_a$ is a set for any uniformly discrete $a\in M \models \MSE$. Furthermore, $x \mapsto |x|_a$ is a definable function on $\Pc(a)$.

\begin{defn}
  Given a uniformly discrete set $a$ and a set $b \sqsubseteq \Pc(a)$, we write $\succc_a(b)$ for the collection $\{c \sqin \Pc(a) : (\exists f \sqin b)[f\sqsubseteq c \wedge \exists ! g(g \sqin f \wedge c \not \sqin f)]\}$.
\end{defn}

It follows from \cref{prop:disc-sep} that $\succc_a(b)$ is a set for any uniformly discrete $a$ and $b \sqsubseteq a$. Note that $\succc_a$ is an explicitly definable function by \cref{lem:uniformly-definable-is-definable-function}. Furthermore, if $b \in \Card_a$, then either $\succc_a(b) \in \Card_a$ or $\succc_a(b) = \varnothing$.

 \begin{defn}
   We write $0$ for the set $\{\varnothing\}$. (Note that $0$ is always an element of $\Card_a$.)
   
   We write $\ind_a(x)$ for the $\Lc_{\dis}(a)$-formula $0 \sqin x \wedge (\forall y \sqin x)\succc_a(y) \in x\vee \succc_a(y) = \varnothing$. If $\ind_a(x)$ holds, we say that $x$ is an \emph{inductive} set.

   We write $\Nb_a$ for the set $\{x \in \Pc^2(a) : (\forall y \in \Pc^3(a))\ind_a(y)\to x \sqin y\}$.
 \end{defn}

It is clear that $\ind_a(\Card_a)$ always holds and so $\Nb_a \sqsubseteq \Card_a$ for any uniformly discrete $a$. 

Now we can finally state the local version of the axiom of infinity.

\begin{defn}
  We write $\Inf(a)$ for the $\Lc_{\dis}(a)$-sentence $|a|_a \not\sqin \Nb_a$.
\end{defn}

It is easy to show that $\Inf(a)$ holds if and only if $\Nb_a \neq \Card_a$.

Now provided that one can find a uniformly discrete $a \in M \models \MSE$ such that $M \models \Inf(a)$, we have that $(\Nb_a,\Pc(\Nb_a),\Pc^2(\Nb_a),\dots)$ is a model of full $\omega$th-order arithmetic, which is more than sufficient to develop ordinary mathematics.

Given the local nature of our development here, one might worry that there could be uniformly discrete $a,b \in M \models \MSE$ for which $\Inf(a)$ and $\Inf(b)$ both hold but $\Nb_a$ and $\Nb_b$ are not internally isomorphic. Fortunately, since our theory is fully impredicative, this cannot happen.

\begin{prop}\label{prop:global-cardinality}
  Fix $M \models \MSE$ and uniformly discrete $a\sqsubseteq b \in M$.
  \begin{enumerate}
  \item The equivalence relation $\approx_a$ is the restriction of the equivalence relation $\approx_b$ to $\Pc(a)\times \Pc(a)$. Write $\iota$ for the induced map from $\Card_a$ to $\Card_b$.
  \item If $M \models \Inf(a)$, then $M \models \Inf(b)$ and $\Nb_b$ is the image of $\Nb_a$ under $\iota$.
  \item For any uniformly discrete $c$, if $M \models \Inf(a) \wedge \Inf(c)$, then $\Nb_a \approx \Nb_c$.
  \end{enumerate}
\end{prop}
\begin{proof}
  1 follows from the fact that any bijection between subsets of $a$ that exists as a subset of $b \times b$ is already a subset of $a \times a$. 2 and 3 follow from 1.
\end{proof}

\section{Global structure of models of $\MSE$}%
\label{sec:global-structure}

While \cref{sec:disc-set} gives a satisfactory picture of the local structure of a model of $\MSE$ around some collection of uniformly discrete sets, the axiom of infinity is a global statement in that it says that there is some set, somewhere, that is infinite.

It is clear that we can take a more global view of cardinality for uniformly discrete sets in a model of $\MSE$. $\approx$ is a perfectly well-defined equivalence relation externally and it would make sense to call its equivalence classes \emph{cardinals}, but we cannot write it as a formula. While $0 = \{\varnothing\}$ and the class of all singletons, $1$, are both sets, the class of doubletons in a model of $\MSE$ is never closed, which precludes it from being a set.\footnote{This phenomenon is also seen in models of $\GPK^+$.}

Suppose furthermore that there is a $\frac{1}{2}$-discrete set $a$ such that $M \models \Inf(a)$. Does this necessarily imply that there is a $1$-discrete set $b$ such that $M \models \Inf(a)$? Can we somehow scale a set up in this way? We will see in \cref{thm:ord-ill} that the answer is no, and so in order to state the axiom of infinity in a global way, we will need to specify the scale at which infinity is to first appear. 

\subsection{Collecting $\e$-discrete sets and `the' axiom of infinity}
\label{sec:collecting-e-discrete}

But before we can formalize that, we need to deal with another subtlety we have been ignoring up until now. How do we even know that we can find \emph{any} uniformly discrete sets? Clearly $\varnothing$ is uniformly discrete, and likewise all hereditarily finite sets are, but it is not clear that the class of hereditarily finite sets is even a set.

What we would like to be able to do is make a formula $\varphi(x)$ that returns $\dis(x)\coloneqq\sup \{r > 0 : x~\text{\rm is}~r\text{\rm -discrete}\}$, but $\dis$ is not a continuous function and so cannot possibly be a formula. If $(a_n)_{n \in \Nb}$ is a Cauchy sequence limiting to $a$ with $a_n \neq a$ for all $n \in \Nb$, then $(\{a_n,a\})_{n\in \Nb}$ will be $d(a_n,a)$-discrete but no better for every $n$, yet the limit, $\{a\}$, will be $1$-discrete. (This is just the fact that the class of doubletons is not closed again.)

This makes it seem unlikely that we will be able to even approximately collect the $r$-discrete sets into a class. Nevertheless, we are able to do something nearly as good.

Fix $r > 0$ and consider the formula
\[
\varphi_r(x) = \sup_{z,y\sqin x}\min(d(y,z),r-d(y,z)).
\]
Note that $\varphi_r^M(a) \leq \e$ if and only if for every $b,c \sqin a$, either $d(y,z) \leq \e$ or $d(y,z) \geq r -\e$. If moreover $\e < \frac{1}{3}r$, this implies that the formula
\[
E_r(x) = \max(\min(\tfrac{1}{r}(2r-3d(x,y)),1),0)
\]
defines a discrete equivalence relation on the elements of $a$. If $a$ is already $r$-discrete, then this equivalence relation is equality. For any $\e > 0$ with $\e < \frac{1}{3}r$, let
\[
  X_{r} = \left\{x \in M : \varphi_r^M(x)< \tfrac{1}{3}r\right\}
\]
and let $f_{r} : X_{r} \to M$ be the function that takes $a$ to the set of $E_r$-equivalence classes of $a$. By \cref{lem:uniformly-definable-is-definable-function} and \cref{prop:disc-eq-reln-quot}, $f_{r}$ is an explicitly definable partial function. Let $\psi_{r}(x,y)$ be a formula defining it (i.e., for any $a \in X_{r}$ and $b \in M$, $\psi_{r}^M(a,b) = d(f_{r}(a),b)$). (Note that $\psi_{r}(x,y)$ does not need any parameters.) Note that for any $a \in X_{r}$, $f_{r}(a)$ is $\frac{1}{3}r$-discrete.

\begin{defn}\label{defn:ax-of-infty}
  For any $r,\e > 0$ with $\e < \frac{1}{3}r$, we let $\Inf_{r,\e}$ denote the condition
  \[
    \sup_x \min\left(1+\e-\varphi_e(x),\cset{\Inf(f_{r}(x))}_{\frac{1}{3}r}\right)\geq 1,
\]
where $\cset{\Inf(f_{r}(x))}_{\frac{1}{3}r}$ means $\inf_z\cset{\Inf(z)}_{\frac{1}{3}r} + 2v(\cset{\Inf(-)}_{\frac{1}{3}r})\psi_{r}(x,z)$.

We let $\Inf$ denote the collection of conditions $\{\Inf_{1,\e} : 0<\e < \frac{1}{3}\}$.
\end{defn}

\begin{prop}\label{prop:ax-infty-char}
  Fix $M \models \MSE$.
  \begin{enumerate}
  \item For any $r,\e > 0$ with $\e < \frac{1}{3}r$, $M \models \Inf_{r,\e}$ if and only if for every $s \in (0,r-\e)$, there is an $s$-discrete set $a \in M$ such that $M \models \Inf(a)$.
  \item $M \models \Inf$ if and only if for every $r \in (0,1)$, there is an $r$-discrete set $a \in M$ such that $M \models \Inf(a)$.
  \end{enumerate}
\end{prop}
\begin{proof}
  2 follows immediately from 1. For the $\To$ direction of $1$, assume that $M \models \Inf_{r,\e}$. This implies that for any $\delta > 0$, there is an $b \in M$ such that $M \models 1 + \e - \varphi_r(b) > 1 - \delta$ and $M \models \inf_z\cset{\Inf(z)}_{\frac{1}{3}r} + 2v(\cset{\Inf(-)}_{\frac{1}{3}r})\psi_{r}(b,z) > 1 - \delta$. The first condition implies that $\varphi_r^M(b) < \e+\delta$. For $\delta < \frac{1}{3}r - \e$, this implies that $b$ is in $X_{r}$ and so $f_r(b)$ is $(r-\e-\delta)$-discrete. Since $b \in X_r$, the second condition is equivalent to $M \models \cset{\Inf(f_{r}(b))}_{\frac{1}{3}r} \geq 1$, which is equivalent to $M \models \Inf(f_r(b))$ and we can take $f_r(b)$ to be the required $a$. Since we can do this for any $\delta > 0$ (and since an $s$-discrete set is $t$-discrete for any $t < s$), we have the required statement.  

For the $\Leftarrow$ direction, fix $s \in (\frac{2}{3}r,r-\e)$ and let $a$ be an $s$-discrete set such that $M \models \Inf(a)$. Now we clearly have that $\varphi_r(a) \leq r-s < \frac{1}{3}r$. Therefore $f_r(a)$ is defined (and equal to $\{\{x\} : x \sqin a\}$). Since $M \models \Inf(a)$, $M \models \Inf(f_r(a))$ as well. (There is an obviously definable bijection between $a$ and $f_r(a)$. This exists as an element of $\Pc(a\times f_r(a))$.) Since we can do this for any sufficiently large $s < r-\e$, we're done.
\end{proof}

By the discussion in \cref{sec:disc-set}, any of the axioms $\Inf_{r,\e}$ is a sufficient form of the axiom of infinity for the purposes of developing standard mathematics. Nevertheless, we propose the scheme $\Inf$ as a canonical choice for `the axiom of infinity' in the context of $\MSE$. One objection to this proposal might be that it is a scheme, rather than a single axiom, but as discussed in \cite[Sec.\ 6.1]{MetSpaUniv}, the concept of finite axiomatizability is murky in continuous logic.

For most of the models we construct in \cref{sec:constructing-models}, there is a $1$-discrete set $a$ for which $\Inf(a)$ holds. This is obviously a more comfortable condition than $M\models \Inf$, but it is unclear whether it is actually axiomatizable. We could achieve it by adding a constant for some such $a$, but this is unsatisfying. Thus we have the following question.

\begin{quest}
  Is the class $\{M \models \MSE : (\exists a \in M)a~\text{\rm is}~1\text{\rm -discrete,}~M\models \Inf(a)\}$ elementary in the sense of continuous logic?
\end{quest}

General pessimism leads us to believe that the answer to this is no, but we do not see an approach to resolving this question.

\subsection{Ordinals}
\label{sec:ord}

Rather than develop the global structure of cardinals in models of $\MSE$, we will focus on ordinals. We do this for a couple of reasons. Many of the technical details for cardinals and ordinals are similar but not quite similar enough to develop simultaneously in an expeditious way. Furthermore, more can be said about the structure of ordinals than of cardinals without assuming some form of the axiom of choice.

\begin{defn}
  Fix $M \models \MSE$. A \emph{chain} in $M$ is a set $a$ such that for any $b,c \in a$, either $b \sqsubseteq c$ or $c \sqsubseteq b$.

  Two uniformly discrete chains $a$ and $b$ are \emph{order-isomorphic} if there is a bijection $f \in \Pc(a\times b)$ such that for any $g,h \sqin a$, it holds that $g \sqsubseteq h$ if and only if $f(g) \sqsubseteq f(h)$. We write $a \cong b$ to signify that $a$ and $b$ are order-isomorphic. The \emph{order type of $a$} is the $\cong$-class of $a$, written $\otp(a)$.

  A uniformly discrete chain $a$ is \emph{well-ordered} if for any non-empty $b \sqsubseteq a$, there is a $\sqsubseteq$-least element of $b$.

  The order types of uniformly discrete well-ordered chains in $M$ are referred to as the \emph{ordinals} of $M$, and the collection of such is written $\Ord^M$.
\end{defn}

Note that we will typically use the term well-ordered to mean internally well-ordered. We will use the word `externally' if we wish to emphasize that something is externally well-ordered.

Given a more general sort of linear order, namely a pair $(a,b)$ with $a$ uniformly discrete and $b \sqsubseteq a\times a$ the graph of a linear order, we can find a uniformly discrete chain $c$ such that $(a,b)$ and $(c,{\sqsubseteq}\res c\times c)$ are internally order-isomorphic. We just need to map each element $f$ of $a$ to the set $\{x \sqin a : \langle x,f \rangle \sqin c\}$ (i.e., the $c$-initial segment with largest element $f$). In this way we can see that uniformly discrete chains are sufficient to represent all uniformly discrete linear order types in models of $\MSE$.

We denote order types of uniformly discrete well-ordered chains with lowercase Greek letters near the beginning of the alphabet, such as $\alpha$ and $\beta$. We write $\alpha \leq \beta$ to mean that for any $a$ with $\otp(a) = \alpha$ and any $b$ with $\otp(b) = \beta$, $a$ is order-isomorphic to some initial segment of $b$. We write $\alpha < \beta$ to mean that $\alpha \leq \beta$ and $\alpha \neq \beta$. By a completely standard argument, we have that for any ordinals $\alpha,\beta \in \Ord^M$, either $\alpha < \beta$, $\beta < \alpha$, or $\alpha = \beta$.

To what extent can we approximate the class of well-ordered uniformly discrete chains with a set? As is typically the case in set theories with a universal set, something fishy needs to happen with regards to the class of ordinals, on pain of the Burali-Forti paradox. In particular, it is immediate that there cannot be a uniformly discrete set containing representatives of all ordinals of $M$.

Using techniques similar to those in \cref{sec:collecting-e-discrete}, we are able to collect representatives of all well-order types occurring below a certain scale. Just as there, we can't easily form sets that consist solely of $r$-discrete chains, only things that are in some sense `approximate chains.' We can use a similar trick, however, to turn these into order-isomorphic chains.

\begin{defn}
  Let $\sigma(x,y) = \sup_{z\sqin x}e(z,y)$. Let $\chn(x) = \sup_{y,z\sqin x}\min(\sigma(y,z),\sigma(z,y))$.
\end{defn}

Note that $\sigma(a,b) = 0$ if and only if $a \sqsubseteq b$. Note also that $d(a,b) = \max(\sigma(a,b),\sigma(b,a))$. Furthermore, $\chn(a) = 0$ if and only if $a$ is a chain. Let $\varphi_r $, $E_r$, $X_r$, and $f_r$ be defined as they were in \cref{sec:collecting-e-discrete}.

\begin{lem}\label{lem:approx-chain}
  For any $r > 0$ and $a \in M \models \MSE$, if $\varphi_r(a)< \frac{1}{3}r$ and $\chn(a) < \frac{1}{3}r$, then for any $b,c\sqin a$, exclusively either $d(a,b) < \frac{1}{3}r$, $\sigma(a,b) > \frac{2}{3}r$, or $\sigma(b,a) > \frac{2}{3}$. 
\end{lem}
\begin{proof}
  Since $\varphi_r(a) < \frac{1}{3}r$, we have that for any $b,c \sqin a$, either $d(b,c) < \frac{1}{3}r$ or $d(b,c) > \frac{2}{3}r$. Since $\chn(a) < \frac{1}{3}r$, we have that for any $b,c \sqin a$, either $\sigma(b,c) < \frac{1}{3}r$ or $\sigma(c,b) < \frac{1}{3}r$. If $d(b,c) \not < \frac{1}{3}r$, then we must have either $\sigma(b,c)\geq \frac{1}{3}r$ or $\sigma(c,b) \geq \frac{1}{3}r$, whence either $\sigma(b,c) > \frac{2}{3}r$ and $\sigma(c,b) < \frac{1}{3}r$ or $\sigma(c,b) > \frac{2}{3}r$ and $\sigma(b,c) < \frac{1}{3}r$. 
\end{proof}

Define the formula
\[
  o(x,y) = \sup_{z \sqin x}\sup_{w \sqin y}\sigma(z,w).
\]
\cref{lem:approx-chain} implies that if $\varphi_r(a) < \frac{1}{3}r$ and $\chn(a) < \frac{1}{3}$, then for any $E_r$-equivalence classes $b$ and $c$ of $a$, either $o(b,c)\leq \frac{1}{3}r$ or $o(b,c)\leq \frac{1}{3}r$. Furthermore, if both of these hold, then $b = c$.

Let $C_r$ be the class $\{x \in X_r : \chn(x) < \frac{1}{3}r\}$. By the above observations, we have that the function $g_r : C_r \to M$ defined by
\[
g_r(a) = \left\{ \{x \sqin f_r(a) : o(x,y) \leq \tfrac{1}{3}r \} :  y \in f_r(a) \right\}
\]
is explicitly definable (without parameters). (Note that we do not need to take closures as $f_r(a)$ is $\frac{1}{3}r$-discrete.) Furthermore, it is immediate that $g_r(a)$ is a $\frac{1}{3}r$-discrete chain for any $a \in C_r$ and if $a \in C_r$ is a chain, then $g_r(a) = \{\{\{x\} : x \sqin a,~x\sqsubseteq y\} : y \in a\}$ and moreover $a$ and $g_r(a)$ are order-isomorphic as chains.

With this machinery, we are finally in a position to examine the global structure of ordinals in models of $\MSE$. In particular, we will show that 

\begin{defn}
  For any ordinal $\alpha$ of $M \models \MSE$, we write $s(\alpha)$ for the quantity
  \[
    \sup\{r> 0 : (\exists~\text{well-ordered}~r\text{-discrete chain}~x\in M)\otp(x) = \alpha\}.
  \]
\end{defn}

It is easy to see that if $\alpha \leq \beta$, then $s(\alpha) \geq s(\beta)$, so for any $M \models \MSE$, $s$ is a non-increasing map from $\Ord^M$ to $(0,1]$. Furthermore, it is always the case that $s(0) = 1$. By using Hartogs numbers, its easy to show that for any ordinal $\alpha \in \Ord^M$, there is an ordinal $\beta \in \Ord^M$ of strictly larger cardinality, namely the Hartogs number of $\Pc(a)$, where $\otp(a) = \alpha$. By an abuse of notation, we'll write this as $\aleph(\Pc(\alpha))$. Note that by \cref{lem:unif-disc-prod-power} and the fact that the Hartogs number of $X$ always embeds into $\Pc^3(X)$, we have that $s(\alpha) = s(\aleph(\Pc(\alpha)))$. This means that if $s(\beta) < s(\alpha)$, then $\beta$ has much larger cardinality than $\alpha$.

We'll write $\omega^M$ for the first limit ordinal in $M$, if it exists. The value of $s(\omega^M)$ is directly related to the axiom of infinity.

\begin{prop}
  Fix $M\models \MSE$. For any $r \in (0,1]$ and $\e \in (0,\frac{1}{3}r)$, $M \models \Inf_{r,\e}$ if and only if $\omega^M$ exists and $r \leq s(\omega^M)$. In particular, $M \models \Inf$ if and only if $\omega^M$ exists and $s(\omega^M) = 1$.
\end{prop}
\begin{proof}
Let $\Nb_a^i$ be the set of initial segments of $\Nb_a$ for some uniformly discrete $a$. It is immediate that $\Nb_a^i$ is a well-ordered chain. If $a$ is $r$-discrete, then we have by \cref{lem:unif-disc-prod-power} that $\Nb_a^i$ is $r$-discrete as well. It is easy to show that $\otp(\Nb_a^i) = \omega^M$ if and only if $M \models \Inf(a)$. Conversely, if $\otp(b) = \omega^M$ for some well-ordered uniformly discrete chain $b$, then $M \models \Inf(b)$. The result now follows from \cref{prop:ax-infty-char}.
\end{proof}

\begin{lem}\label{lem:close-implies-iso}
  Let $a$ and $b$ be $r$-discrete chains in some $M\models \MSE$. If $d(a,b) < \frac{1}{2}r$, then $a \cong b$. Furthermore, if $(M,d)$ is an ultrametric space, it is enough to assume that $d(a,b) < r$.
\end{lem}
\begin{proof}
  Fix $s$ such that $d(a,b) < s < \frac{1}{2}r$.
  Since $a$ and $b$ are $r$-discrete, we have that the class $f = \{\langle x,y \rangle \sqin a \times b : d(x,y) \leq s\}$ is a set and is the graph of a bijection between $a$ and $b$. Now we need to show that $f$ is actually an order isomorphism between $a$ and $b$. Suppose that we have $c,c' \sqin a$ with $c \sqsubseteq c'$ and $g,g' \sqin b$ with $d(c,g) < s$ and $d(c',g') < s$. If $c = c'$, then $g = g'$, so assume that $c \sqsubset c'$. Since $d(c,c') \geq r$, we can find an $h \sqin c'$ such that $e(h,c) > 2s$ (as $2s < r$). Since $d(c',g') < s$, we can find an $i \sqin g'$ such that $d(h,i) < s$. The triangle inequality implies that $e(i,g) \geq e(h,c) - d(h,i) - d(c,g) > 2s - s - s = 0.$ Therefore $i \not \sqin g$ and it must be the case that $g \sqsubset g'$, as required.

The proof in the ultrametric case is essentially the same.
\end{proof}

\begin{thm}\label{thm:ord-ill}
  Fix $M \models \MSE$ and $r \in \{s(\gamma) : \gamma \in \Ord^M \}$. Let $t = r$ if $d$ is an ultrametric and let $t = \frac{1}{2}r$ otherwise.

  For any $s < t$, there is an ordinal $\alpha \in \Ord^M$ with $s \leq s(\alpha) < r$ such that for any $\beta \in \Ord^M$, if $s(\beta) \geq r$, then $\beta < \alpha$.

  In particular, if $d$ is an ultrametric, then $\{s(\alpha) : \alpha \in \Ord^M\}$ is dense in $(0,1]$ and $\{\alpha \in \Ord^M : s(\alpha) < 1\}$ has no least element.
\end{thm}
\begin{proof}
  
  Fix positive $s < t$ and let $\delta = 1-\frac{s}{t}$ (implying that $s=t(1-\delta)$). Assume without loss of generality that $t\delta < \frac{1}{3}r$. Let
  \[
    a = \exc{x}{\max(\phi_r(x),\chn_r(x))}{0}{t\delta}.
  \]
  Note that every element of $a$ is an element of $C_r$. Let $b = \ol{\{g_r(x) : x\sqin a\}}$. Note that since the collection of chains in $M$ is $\{x \in M : \chn^M(x) = 0\}$, it is metrically closed. Hence every element of $b$ is a chain. It is also easy to see that every element of $b$ is $r(1-\delta)$-discrete (regardless of whether $d$ is an ultrametric).

  Let $c = \{ x \sqin b : x~\text{is well-ordered}\}$. Note that $c$ is a set in $M$. Also note that for any $\beta \in \Ord^M$, if $s(\beta) \geq r$, then some element of $c$ has order type $\beta$. The equivalence relation $\cong$ is discretely definable on $c$, so we can form the set $f = \{\{y \sqin c : x \cong y\} : x \sqin c\}$. By \cref{lem:close-implies-iso}, the set $f$ is $t(1-\delta)$-discrete (regardless of whether $d$ is an ultrametric) or, in other words, $s$-discrete. We can find a formula $\varphi(x,y)$ with the property that for any $x,y \sqin f$, $\varphi(x,y) \in \{0,1\}$ and $\varphi(x,y) = 0$ if and only $(\forall z \sqin x)(\forall w \sqin y)\otp(z)\leq\otp(x)$. This formula defines a linear order on $f$ which by a standard argument is a well-order. Let
  \[
    \alpha = \otp(\{\{y : y \sqin f,~\varphi(x,y) = 0\} : x\sqin f\}).
  \]
  Since $f$ contains representatives of all ordinals $\beta$ with $s(\beta) \geq r$, we must have that $\alpha$ is larger than any such $\beta$. Therefore it cannot be the case that $s(\alpha) \geq r$. On the other hand, since $f$ is $s$-discrete, it follows that $s(\alpha) \geq s$, as required.

 The last statements in the theorem obviously follow from the rest of it.
\end{proof}

\begin{cor}
  If $M\models \MSE$ has no infinite, uniformly discrete sets, then $M$ has non-standard naturals.
\end{cor}
\begin{proof}
  By \cref{cor:basic-stuff} and induction, any model of $\MSE$ contains hereditarily finite sets of every externally finite cardinality. Therefore for any standard natural $n$, $n \in \Ord^M$ and $s(n) = 1$. \cref{thm:ord-ill} implies that there are ordinals $\alpha$ in $M$ (which must be internally finite) such that $s(\alpha) < 1$. 
\end{proof}

The behavior of ultrametric models of $\MSE$ in \cref{thm:ord-ill} is reminiscent of the behavior of $\omega$ in models of Cantor-\L ukasiewicz set theory, as discovered by H\'ajek \cite[Th.\ 4.17]{Hajek-2015}. In particular, they both exhibit a manifestation of the sorites paradox: an inability to formalize an induction principle of the form
\begin{prooftree}
  \AxiomC{$\varphi(0) = 0$}
  \AxiomC{$\forall \alpha[(\forall \beta < \alpha)(\varphi(\beta) = 0) \to \varphi(\alpha) = 0]$}
  \BinaryInfC{$\forall \alpha(\varphi(\alpha) = 0)$}
\end{prooftree}
for a real-valued predicate $\varphi(x)$ on some class of ordinals. For $ \CL_0$, this induction principle cannot hold even for $\omega$, but, as we will see in \cref{thm:MSE-consistent}, models of $\MSE$ can have arbitrarily large standard ordinals. \cref{thm:ord-ill} is also of course similar to the non-existence of $\beta$-models of $\NFU$, although the mechanism by which models of $\NFU$ are ill-founded is different. One might idly wonder what could happen if we were to restrict excision to stratified formulas.

What is unclear at the moment is the status of non-ultrametric models of $\MSE$. \cref{thm:ord-ill} does not preclude the possibility of $\beta$-models of $\MSE$ (i.e., models in which $\Ord^M$ is externally well-founded), but it seems unlikely that they exist. Every model of $\MSE$ we know how to produce contains a set that is an ultrametric model of $\MSE$, whereby \cref{thm:ord-ill} applies. This leaves the following question.

\begin{quest}
  Does $\MSE$ have any $\beta$-models? Is it true that for any $M \models \MSE$, $\{s(\alpha) : \alpha \in \Ord^M\}$ is dense in $(0,1]$?
\end{quest}

Given the behavior of the models constructed in \cref{sec:constructing-models}, we conjecture that $\MSE$ has no $\beta$-models and $\{s(\alpha) : \alpha \in \Ord^M\}$ is always dense in $(0,1]$.

Finally, although this is more or less a cosmetic nicety, we would like to show that we can build canonical `tokens' representing well-order types, i.e., elements of $M$ that somehow canonically represent a well-order type $\alpha$. In $\NFU$, this is accomplished by taking $\{x : \otp(x) = \alpha\}$. In $\ZF$ and $\GPK^+_\infty$, this is accomplished by taking the von Neumann ordinal of that order type. Neither of these approaches will work for $\MSE$, so we will have to do something new. %

\begin{defn}\label{defn:ord-token}
  Given a uniformly discrete set $a$, a set $b \sqsubseteq a \times V^M$, and an element $c \sqin a$, we write $b[c]$ for the class $\{f : \langle c,f \rangle \sqin b\}$. For any chain $a \in M \models \MSE$ and any set $b \sqsubseteq a \times V^M$, the \emph{closed chain union of $b$} is 
  \[
    \chi(b) \coloneqq \overline{\left\{ \overline{\bigsqcup\{b[f]:f\sqin a,~f\sqsubseteq c\}} : c \sqin a \right\}}.
  \]
  
  Given a uniformly discrete chain $a \in M \models \MSE$, the \emph{order token of $a$} is the class
  \[
    \otok(a) \coloneqq \overline{\{\chi(b) : b \sqsubseteq a \times V^M\}}.
  \]
\end{defn}

\begin{thm}\label{thm:canonical-ordinals}
  Fix $M \models \MSE$
  \begin{enumerate}
  \item $\otok(a)$ is a set for any uniformly discrete chain $a$.
  \item For any $r > 0$, the map $x \mapsto \otok(a)$ is explicitly definable on the class of $r$-discrete chains.
  \item If $a$ and $b$ are well-ordered uniformly discrete chains, then $\otp(a) \leq \otp(b)$ if and only if $\otok(a) \sqsubseteq \otok(b)$.
  \end{enumerate}
\end{thm}
\begin{proof}
  1 and 2 follow from \cref{prop:defbl-omnibus} and \cref{lem:im-of-def-is-def}.
  
  For 3,  assume that $\otp(a) \leq \otp(b)$. Let this be witnessed by an order isomorphism $f: a \to b$ to some initial segment of $b$. For any $c \sqsubseteq a \times V^M$, we can form the set $c = \{\langle f(x),y \rangle : \langle x, y \rangle \sqin c\}$ (because the map $\langle x,y \rangle \mapsto \langle f(x),y \rangle$ is definable) and we immediately have that $\chi(c) = \chi(c_f)$. Therefore $\otok(a) \sqsubseteq \otok(b)$.

  Conversely, assume that $\otok(a) \sqsubseteq \otok(b)$. Let $r$ be such that $a$ and $b$ are $r$-discrete. Find $c \sqsubseteq b \times V^M$ such that $d(a,\chi(c)) < \frac{1}{2}r$. By \cref{lem:close-implies-iso}, we have that $a$ and $\chi(c)$ are order-isomorphic as chains. Let this be witnessed by $f : a \to \chi(c)$. For any $x \sqin a$, let $g(x)$ be the smallest element of $b$ such that $f(x) \sqsubseteq \ol{\bigcup \{c[z]:z\sqin b,~z\sqsubseteq g(x)\}}$. (This is a set by \cref{prop:disc-sep} and the fact that $a$ and $b$ are $r$-discrete.) $g$ is an injective order-preserving map from $a$ to $b$, so by a standard argument, we have that $\otp(a) \leq \otp(b)$.
\end{proof}

Now of course, given $\otok(a)$, we can build a canonical well-ordered chain with the same order type as $a$, namely 
\[
 \ord(a)\coloneqq \{\otok(z) : z~\text{is a well-ordered uniformly discrete chain},~\otp(z)<\otp(a)\}. 
\]
One can show that $x \mapsto \ord(x)$ is explicitly definable on the class of $r$-discrete well-ordered chains for any $r > 0$.

Naturally, we could attempt to do something similar to \cref{defn:ord-token} with cardinalities, but without some form of the axiom of choice, we only seem to be able to build tokens representing equivalence classes of the $\approx^\ast$ relation (where $x \leq^\ast y$ if there is a surjection from some subset of $y$ onto $x$ and $x \approx^\ast y$ if $x\leq^\ast y$ and $y \leq^\ast x$). Specificially, if we define $\ctok^\ast(a)\coloneqq \ol{\{\ol{\{\pi_1(x):x \sqin y\}} : y \sqsubseteq a \times V^M\}}$, we then have that $a \approx^\ast b$ if and only if $\ctok^\ast(a) = \ctok^\ast(b)$ for any uniformly discrete $a$ and $b$. This raises an obvious question.

\begin{quest}
  Is there a function $\ctok(x)$ that is definable on the class of $r$-discrete sets for each $r > 0$ such that for any uniformly discrete $a$ and $b$, $a \approx b$ if and only if $\ctok(a) = \ctok(b)$?
\end{quest}

\section{Formalizing $\MSE$ in continuous logic}
\label{sec:formalizing-in-CL}

Given a metric set structure $(M,d,\sqin)$, we can build a \emph{general structure}\footnote{As defined in \cite{2005.11851}. Such structure could also be described as \emph{metric structures without a metric}.} $(M,e)$ by taking the function $e:M^2 \to [0,1]$ to be the sole predicate. After doing so, the original structure can be recovered by taking $d(x,y) = \sup_z|e(z,x)-e(z,y)|$ and ${\sqin} = \{(x,y) \in M^2: e(x,y) = 0\}$. Our goal in this section is to characterize the structures that arise in this way and show that they form an elementary class in the sense of continuous logic.

Let $\Lc_e$ be the language with a single $[0,1]$-valued predicate symbol $e$. Given any $\Lc_e$-structure $(M,e)$, we can define a pseudo-metric
\[
d_e(x,y) \coloneqq \sup_z|e(z,x)-e(z,y)|.
\]
Since this is a formula in the sense of continuous logic, $d_e$ is a definable predicate on any $\Lc_e$-structure. Note that by construction, for any $\Lc_e$-structure $(M,e)$ and $a \in M$, the function $y \mapsto e(a,y)$ is $1$-Lipschitz with regards to $d_e$.

The first thing we need to do is write out an axiom that guarantees that $e(x,b)$ is a distance predicate with regards to $d_e$ for any choice of $b$. This is implicitly done in \cite[Ch.~9]{MTFMS}, but the characterization of distance predicates there does not cover the possibility of an empty definable set. This is easy enough to add in by hand, but we will take the opportunity to make the paper more self-contained and give a cleaner\footnote{The proof of \cref{lem:ax-H-ext-char} implicitly contains a proof of the following more general fact: In any metric structure $M$ with a $[0,r]$-valued metric, a formula $\varphi(x)$ is the distance predicate of a (possibly empty) definable set if and only if $M\models \sup_{x}|\varphi(x) - \inf_z \min(d(x,z) + 2\varphi(z),r)|$. This is a slight modification of the condition $E_2$ in \cite[Ch.~9]{MTFMS} that obviates the need for $E_1$.} axiomatization that also covers both cases.

\begin{defn}\label{defn:H-ext-ax}
  The \emph{$\HH$-extensionality axiom} is the $\Lc_e$-condition
  \[
    \sup_{xy}|e(x,y) - \inf_z \min(d_e(x,z) + 2e(z,y),1)| = 0.
  \]
  We say that an $\Lc_e$-structure $M$ is \emph{$\HH$-extensional} if it satisfies the $\HH$-extensionality axiom.
\end{defn}
Note that the $\HH$-extensionality axiom could more conventionally be written
\[
\forall x \forall y (e(x,y) = \inf_z \min(d_e(x,z)+2e(z,y),1)).
\]

Given an $\Lc_e$-structure $(M,e)$, we write $M/e$ for the $d_e=0$ quotient of $M$ and we write $\ol{M/e}$ for the completion of this under $d_e$. Given $a \in M$, we write $[a]_e$ for the corresponding element of $M/e$, which we regard as a subset of $\ol{M/e}$.

\begin{lem}\label{lem:ax-H-ext-char}
Fix an $\Lc_e$-structure $(M,e)$. $M$ satisfies the $\HH$-extensionality axiom if and only if for any $b \in M$, the function $x \mapsto e(x,b)$ is $1$-Lipschitz with regards to $d_e$ and if $f(x)$ is the extension of $e(x,b)$ to $\ol{M/e}$, then for any $a \in \ol{M/e}$, $f(a) = \inf\{d_e(a,c) : c \in \ol{M/e},~f(c) = 0\}$, where $\inf \varnothing = 1$.
\end{lem}
\begin{proof}
  For the $\To$ direction, suppose that $M$ satisfies the $\HH$-extensionality axiom and fix $b \in M$. We have by the $\HH$-extensionality axiom that $e(x,b) = \inf_z \min(d_e(x,z)+2e(z,b),1)$ for all $x$. The function $x \mapsto \min(d_e(x,z)+2e(z,b),1)$ is $1$-Lipschitz with regards to $d_e$ for any $z$, therefore $e(x,b)$ is as well (since it is the infimum of a family of $1$-Lipschitz functions). Let $f(x)$ be the extension of $e(x,b)$ to $\ol{M/e}$.

  Fix $a \in \ol{M/e}$ such that $f(a) < 1$. Also fix $\e > 0$. Since $M/e$ is dense in $\ol{M/e}$, we can find $c_0 \in M$ such that $d_e(a,[c_0]_e) < \frac{1}{3}\e$ and $f([c_0]_e) = e(c_0,b) < 1$. Note that $f(a) < f([c_0]_e) = e(c_0,b) + \frac{1}{3}\e$.  Fix $\delta > 0$ with $\delta < \frac{1}{4}\e$, $\delta < e(c_0,b)$, and $e(c_0,b) + \delta < 1$.

  Suppose we are given $c_n \in M$ with $e(c_n,b) \leq e(c_0,b) < 1$ for some $n\in \Nb$. If $e(c_n,b) = 0$, stop the construction and set $c_m = c_n$ for all $m > n$, otherwise we can find $c_{n+1} \in M$ such that
  \[
    d_e(c_n,c_{n+1}) + 2e(c_{n+1},b) < e(c_n,b) + \min(4^{-n}\delta,\tfrac{1}{2}e(c_0,b))
  \]
  by the $\HH$-extensionality axiom. In particular, this implies that 
  \begin{align*}
    e(c_{n+1},b) &< \frac{1}{2}e(c_n,b) + \frac{1}{2}e(c_0,b) \\
                 &\leq e(c_0,b) < 1.
  \end{align*}

  We have at each $n$ that $e(c_{n+1},b) < \frac{1}{2}e(c_n,b) + 4^{-n}\delta$. Recursively applying this bound gives
  \[
    e(c_{n},b) < 2^{-n}e(c_0,b) + \sum_{k=n-1}^{2n-2}2^{-k}\delta.
  \]
  for any $n > 0$ with $c_n$ defined. This implies that the infinite sums we are about to manipulate are all absolutely convergent and that $(c_n)_{n\in \Nb}$ is a Cauchy sequence.\footnote{Since $\sum_{k=n-1}^{2n-2}2^{-k} = 2^{-n+2}-2^{-2n+2} < 2^{-n+2}$, we have that $e(c_n,b) < 2^{-n}(e(c_0,b) + 4\delta)$. This in turn implies that for any $n$ at which $c_{n+1}$ is defined, $d_e(c_n,c_{n+1}) < 2^{-n}(e(c_0,b) + 4\delta) + 4^{-n}\delta$.} We have that
  \begin{align*}
    \sum_{n=0}^\infty \left( d_e(c_n,c_{n+1})+ 2e(c_{n+1},b)  \right)&< \sum_{n=0}^\infty \left( e(c_n,b) + 4^{-n}\delta \right), \\
    \sum_{n=0}^\infty d_e(c_n,c_{n+1})+ 2\sum_{n=0}^\infty e(c_{n+1},b) &< \sum_{n=0}^\infty e(c_n,b) + \sum_{n=0}^\infty4^{-n}\delta, \\
    \sum_{n=0}^\infty d_e(c_n,c_{n+1})+ \sum_{n=0}^\infty e(c_{n+1},b) + e(c_{\infty+1},b) &< e(c_0,b) + \frac{4}{3}\delta, \\
    \sum_{n=0}^\infty d_e(c_n,c_{n+1}) &< e(c_0,b) + \frac{4}{3}\delta. 
  \end{align*}
 Let $g = \lim_{n\to \infty}[c_n]_e$. Note that $e(g,[b]_e) = 0$. We now have that
  \begin{align*}
    d_e(a,g) &\leq d_e(a,[c_0]_e) + d_e([c_0]_e,g) \\
             & < \frac{1}{3}\e + \sum_{n=0}^\infty d_e(c_n,c_{n+1}) \\
             & < \frac{1}{3}\e + e(c_0,b) + \frac{4}{3}\frac{1}{4}\e \\
             & < \frac{1}{3}\e + f(a) + \frac{1}{3}\e + \frac{1}{3}\e \\
    & = f(a) + \e.
  \end{align*}
  Since we can do this for any $\e > 0$, we have that $\inf\{d_e(a,x):f(x) = 0\} \leq f(a)$ whenever $f(a) < 1$. If $f(a) = 1$, then this inequality holds anyway, so the inequality holds in all cases.

  For the other inequality, first assume that $\{x \in \ol{M/e}:f(x) = 0\}$ is empty. By the previous part, this implies that $f(a) = 1$ for all $a \in \ol{M/e}$. Therefore the required equality holds. Now assume that $\{x \in \ol{M/e}:f(x) = 0\}$ is non-empty. Fix $a \in \ol{M/e}$ and let $r = \inf\{d_e(a,x) : f(x) = 0\}$. Fix $\e > 0$ and find some $a' \in M$ such that $d_e(a,[a']_e) < \frac{1}{6}\e$. Since $f$ is $1$-Lipschitz, we must have that $f(a) < f([a']_e) + \frac{1}{6}\e = e(a',b) + \frac{1}{6}\e$.

  Find some $c \in \ol{M/e}$ such that $f(c) = 0$ and $d_e(a,c) < r + \frac{1}{3}\e$. Find $g \in M$ such that $d_e(c,[g]_e) < \frac{1}{3}\e$. Since $f(x)$ is $1$-Lipschitz, we have that $f([g]_e) = e(g,b) < \frac{1}{3}\e$. By the $\HH$-extensionality axiom, $e(a',b) \leq d_e(a',c') + 2e(c',b)$, so we have that
  \begin{align*}
    f(a) & < e(a',b) + \frac{1}{6}\e \\
         & \leq d_e(a',c') + 2e(c',b) + \frac{1}{6}\e \\
         &< d_e(a',c') + 2\frac{1}{6}\e + \frac{1}{6} \e  \\
         &\leq d_e(a,c) + \frac{2}{6}\e + 2\frac{1}{6}\e + \frac{1}{6} \e  \\
         &< r + \frac{1}{6}\e + \frac{2}{6}\e + 2\frac{1}{6}\e + \frac{1}{6} \e  \\
         &= r+ \e. 
  \end{align*}
  Since we can do this for any $\e > 0$, we have that $f(a) \leq r = \inf\{d_e(a,x):f(x) = 0\}$. Therefore both directions of the inequality hold and we have that $f(a) = \inf\{d_e(a,x):f(x) = 0\}$ for any $a \in \ol{M/e}$.

  For the $\Leftarrow$ direction, suppose that $e(x,b)$ is $1$-Lipschitz with regards to $d_e$ for any $b \in M$ and that for any $a \in \ol{M/e}$, $f(a) = \inf\{d_e(a,x):f(x)= 0\}$, where $f(x)$ is the unique continuous extension of $e(x,b)$ to $\ol{M/e}$. Fix $a \in M$ and let $r = e(a,b) = f([a]_e)$. Fix $\e > 0$. Find $c \in \ol{M/e}$ such that $f(c) = 0$ and $d([a]_e,c) < r + \frac{1}{4}\e$. Find $c' \in M$ such that $d(c,[c']_e) < \frac{1}{4}\e$. Since $f(x)$ is $1$-Lipschitz, we have that $f([c']_e)=e(c',b) < \frac{1}{4}\e$. Note also that $d_e(a,c') < r + \frac{2}{4}\e$. We now have that
  \begin{align*}
    \inf_z \min(d_e(a,z) + 2e(z,b),1) & \leq d_e(a,c') + 2e(c',b) \\
                                      &\leq r + \frac{2}{4}\e + 2e(c',b) \\
                                      &\leq r + \frac{2}{4}\e + 2\frac{1}{4}\e \\
                                      &\leq r + \e.
  \end{align*}
  Since we can do this for any $\e > 0$, we have that
  \[
    \inf_z \min(d_e(a,z)+2e(z,b), 1) \leq \inf\{d_e(a,x) : f(x) = 0\} = e(a,b).
  \]
  For the other direction of the inequality, let $s = \inf_z \min(d_e(a,z)+2e(z,b),1)$. If $s = 1$, then the above implies that $f(x) = 1$ for all $x \in \ol{M/e}$, so the $\HH$-extensionality axiom is satisfied. Otherwise assume that $s < 1$ and fix $\e > 0$ with $s + \e < 1$. Find $c \in M$ such that $\min(d_e(a,c)+2e(c,b),1) < s + \e$. We must have that $d_e(a,c) + 2e(c,b) < s+\e$. By assumption, $e(c,b) = f([c]_e) = \inf\{d_e([c]_e,x) : f(x) = 0\}$. Therefore
  \begin{align*}
  e(a,b) = \inf\{d_e([a]_e,x) : f(x) = 0\} & \leq d_e(a,c) + e(c,b) \\
                                    & < d_e(a,c) + 2e(c,b) \\
                                    &< s + \e.
  \end{align*}
Since we can do this for any $\e > 0$, we have that $e(a,b) \leq s = \inf_z \min(d_e(a,z) + 2e(z,b),1)$. Therefore $e(a,b) = \inf_z\min(d_e(a,z) + 2e(z,b)),1$ for any $a,b \in M$ and the $\HH$-extensionality axiom holds. 
\end{proof}

Note that since $y \mapsto e(a,y)$ is automatically $1$-Lipschitz with regards to $d_e$, the $\HH$-extensionality axiom implies that $(x,y) \mapsto e(x,y)$ is $2$-Lipschitz with regards to $d_e$. This means that $e(x,y)$ extends to a unique continuous function on $\ol{M/e}$. By an abuse of notation we will also denote this as $e$. Note that in this case, $(\ol{M/e},e)$ still satisfies the $\HH$-extensionality axiom (and is in fact elementarily equivalent to $(M,e)$ as an $\Lc_e$-structure). In particular, by \cref{lem:ax-H-ext-char} applied to the structure $(\ol{M/e},e)$, we have that $x\mapsto e(x,b)$ is a distance predicate for any $b \in \ol{M/e}$.

Given an $\Lc_e$-structure $M$ for which $e(x,y)$ extends to $\ol{M/e}$, write $\sqin_e$ for the relation $\{(x,y) \in (\ol{M/e})^2 : e(x,y) = 0\}$. Now we will see the manner in which the $\HH$-extensionality axiom characterizes metric set structures.

\begin{prop}
  Fix an $\Lc_e$-structure $(M,e)$. $(M,e)$ satisfies the $\HH$-extensionality axiom if and only if $(\ol{M/e},d_e,\sqin_e)$ exists and is a metric set structure.
\end{prop}
\begin{proof}
  This follows from \cref{lem:ax-H-ext-char} and the fact that $(d_e)_{\HH}(\{x:x\sqin_e a\},\{x:x\sqin_e b\}) = \sup_z |e(z,a)-e(z,b)| = d_e(a,b)$ for all $a,b \in \ol{M/e}$.
\end{proof}

Axiomatizing excision will be more technical. For convenience, we'll take restricted $\Lc_e$-formulas to be defined as in \cite[Sec.~1.3]{HansonThesis}: the only atomic formulas are those of the form $e(x,y)$ and we take as connectives $\varphi+\psi$, $\max(\varphi,\psi)$, $\min(\varphi,\psi)$, the constant $1$, and $r\cdot \varphi$ for rational $r$. We should note though that the scheme described here would be sufficient with any definition of restricted formula, such as the one in \cite[Sec.~3]{MTFMS}. 

Given a restricted $\Lc_e$-formula $\varphi$, we can form a corresponding $\Lsq$-formula by replacing each instance of $e(x,y)$ with the $\Lsq$-formula $\inf_{z\sqin y}d(x,z)$ where $z$ is taken to be any variable distinct from $x$ and $y$. We write $\varphi_{\sqin}$ for the formula resulting from this translation. By an abuse of notation, we will write $v(\varphi)$ for $v(\varphi_{\sqin})$.

Later on, we will also need a way to translate $\Lsq$-formulas back to restricted $\Lc_e$-formulas. The difficulty here is that we allowed real coefficients in $\Lsq$-formulas but only rational coefficients in $\Lc_e$-formulas. With this issue in mind say that an $\Lsq$-formula is \emph{rational} if all coefficients occurring in it are rational numbers. We define the $\Lc_e$-formula $\varphi_e$ corresponding to a rational $\Lsq$-formula $\varphi$ inductively as follows:
\begin{itemize}
\item $(d(x,y))_e = d_e(x,y)$,
\item $(\inf_{x\sqin y}\varphi)_e = \inf_x\min(\varphi_e + 2v(\varphi)e(x,y),v(\varphi))$, and
\item $(\sup_{x\sqin y}\varphi)_e = -(\inf_{x\sqin y} - \varphi)_e$,
\end{itemize}
with the other elements of the translation defined in the obvious way. The following facts are either standard results in continuous logic or easily verified.

\begin{fact}\label{fact:translation}
  Fix an $\HH$-extensional $\Lc_e$-structure $(M,e)$ with $(M,d_e)$ complete.
  \begin{enumerate}
  \item For any $\Lc_e$-formula $\varphi(\xbar)$ and any $\abar \in M$, $\varphi^{(M,e)}(\abar) = \varphi_{\sqin}^{(M,d_e,\sqin_e)}(\abar)$.
  \item For any rational $\Lsq$-formula $\varphi(\xbar)$ and any $\abar \in M$, $\varphi^{(M,d_e,\sqin_e)}(\abar) = \varphi^{(M,e)}_e(\abar)$.
  \item For any $\Lsq$-formula $\varphi(\xbar)$ and $\e > 0$, there is a rational $\Lsq$-formula $\psi(\xbar)$ such that $|\varphi^{(M,d_e,\sqin_e)}(\abar) - \psi^{(M,d_e,\sqin_e)}(\abar)| < \e$ for all $\abar \in M$.
  \end{enumerate}
\end{fact}

It follows from \cref{lem:formula-v} and \cref{fact:translation} that if $(M,e)$ is $\HH$-extensional with $(M,d_e)$ complete, then for any restricted $\Lc_e$-formula $\varphi(\xbar)$, the function $\xbar \mapsto \varphi^M(\xbar)$ is $2v(\varphi)$-Lipschitz with regards to the max metric on tuples induced by $d_e$. By passing to the completion $(\ol{M/e},d_e)$, this implies the same for any $\HH$-extensional $(M,e)$. 

For any formula $\varphi \in \Lc_e$, we define
\[
  \e_\varphi \coloneqq \frac{1}{\max(6v(\varphi),3)}.
\]
 Note that for any $\HH$-extensional $(M,e)$, if $|\varphi^M(\abar) - \varphi^M(\bbar)| \geq \frac{1}{2}$, then $d_e(\abar,\bbar) > \e_\varphi$. %

\begin{defn}\label{defn:ax-sch-exc}
  The \emph{axiom scheme of excision} is the collection of $\Lc_e$-conditions of the form
  \[
    \sup_{\ybar}\inf_z \sup_x \max(\min(e(x,z),-\varphi(x,\ybar)),\min(\e_\varphi-e(x,z),\varphi(x,\ybar) - 1)) \leq 0
  \]
  for each restricted $\Lc_e$-formula $\varphi(x,\ybar)$ (not containing $z$ as a free variable).

  Given an $\HH$-extensional $\Lc_e$-structure $M$, we say that $M$ \emph{satisfies $\Lc_e$-excision} to mean that $M$ satisfies the axiom scheme of excision.
\end{defn}

The axiom scheme of excision can be more conventionally stated like this: For all $\ybar$,  $\delta > 0$, and $\varphi(x,\ybar) \in \Lc_e$, there is a $z$ such that for all $x$,
\begin{itemize}
\item if $\varphi(x,\ybar) \leq -\delta$, then $e(x,z) < \delta$ and
\item  if $e(x,z) \leq \e_\varphi - \delta$, then $\varphi(x,\ybar) < 1+\delta$. 
\end{itemize}
It is also sufficient to assume merely that this holds for sufficiently small $\delta > 0$. This is clearly an approximation of a certain case of the excision principle in $\MSE$, but we will now show that in $\HH$-extensional $M$ with $(M,d_e)$ complete, the axiom scheme of excision is enough to imply full excision.

\begin{lem}\label{lem:fat-comp}
  Fix an $\HH$-extensional $\Lc_e$-structure $M$ with $(M,d_e)$ complete. Suppose that $M$ satisfies $\Lc_e$-excision. For any $a \in M$ and $r,s \in [0,1]$ with $r<s$, there is a $b \in M$ such that for any $c \in M$, if $e(c,a)\leq r$, then $c \sqin_e b$, and if $c \sqin_e b$, then $e(c,a) < s$.
\end{lem}
\begin{proof}
  For readability, we will write $d$ for $d_e$ and $\sqin$ for $\sqin_e$.
  
  Let $r_0 = \frac{2}{3}r+\frac{1}{3}s$, $s_0 = \frac{1}{3}r + \frac{2}{3}s$, and $b_0 = a$. For any $n$, let $\varphi_n(x,y) = \frac{e(x,y) - r_n}{s_n - r_n}$.

  At stage $n$, suppose we are given $b_n$ and rationals $r_n$ and $s_n$ with $0 < r_n < s_n$. Since $M$ satisfies $\Lc_e$-excision, we have that for any $\gamma > 0$, there is an $f \in M$ such that
  \[
\forall x (e(x,f) < \gamma \vee \varphi_n(x,b_n) > -\gamma) \wedge (\e_{\varphi_n} - e(x,b_n) < \gamma \vee \varphi_n(x,b_n) - 1 < \gamma).
\]
Let $b_{n+1}$ be such an $f$ with
\[
  \gamma = \delta_n \coloneqq \min\left(\frac{2^{-n-2}}{27}(s-r),\frac{1}{2}r_n,\e_{\varphi_n} \right).
\]

We have that for any $c \in M$, if $\frac{e(c,b_n)-r_n}{s_n-r_n} \leq -\delta_n$ (i.e., if $e(c,b_n) \leq r_n -\delta_n(s_n-r_n)$), then $e(c,b_{n+1}) < \delta_n$. A fortiori, this implies that if $e(c,b_n) \leq r_n - \delta_n$, then $e(c,b_{n+1}) < \delta_n$.

On the other hand, if $\e_{\varphi_n} - e(c,b_{n+1}) \geq \delta_n$ (i.e., if $e(c,b_{n+1}) \leq \e_{\varphi_n} - \delta_n$), then $\frac{e(c,b_n)-r_n}{s_n-r_n} - 1 < \delta_n$ and so $e(c,b_n) < r_n + (1+\delta_n)(s_n-r_n)$. Since $\delta_n \leq \e_{\varphi_n}$, this implies that if $c \sqin b_{n+1}$, then $e(c,b_n) < r_n + 2s_n$.

Finally, pick $r_{n+1}$ and $s_{n+1}$ so that $2\delta_n < r_{n+1} < s_{n+1} < 3\delta_n$, and move to the next stage of the construction.
\vspace{0.5em}

\noindent\emph{Claim.} $(b_n)_{n<\omega}$ is a Cauchy sequence.
\vspace{0.25em}

\noindent\emph{Proof of claim.} For any $n > 0$ and $c \in M$, we have that if $c \sqin b_n$, then $e(c,b_{n+1}) < \delta_n \leq \frac{2^{-n-2}}{27}(s-r)$ and also that if $c \sqin b_{n+1}$, then $e(c,b_n) < r_n + 2s_n < 9\delta_{n-1}\leq \frac{2^{-n-1}}{3}(s-r)$. Therefore, by $\HH$-extensionality, $d(b_n,b_{n+1}) \leq \frac{2^{-n-1}}{3}(s-r)$. Since we can do this for any positive $n$, the claim follows. \hfill $\square_{\text{claim}}$
\vspace{0.25em}

Let $b = \lim_{n\to \infty}b_n$. $b$ is an element of $M$ since $(M,d)$ is complete.
\vspace{0.5em}

\noindent\emph{Claim.} For any $c \in M$, if $e(c,a) \leq r$, then $c \sqin b$. 
\vspace{0.25em}

\noindent\emph{Proof of claim.} Since $e(c,a) \leq r$ and since $\delta_0 \leq \frac{2^{-2}}{27}(s-r) < \frac{1}{3}(s-r)$, we have that $e(c,b_0) = e(c,a) < r_0 - \delta_0$. Therefore, $r(c,b_1) < \delta_0$. For any $n$, suppose that we know that $e(c,b_{n+1}) < \delta_{n}$. We then have that
\begin{align*}
  e(c,b_{n+1}) &<\delta_{n} & \\
               &< \tfrac{1}{2}r_{n+1} &(\text{by our choice of }r_{n+1}) \\
               &\leq  r_{n+1} - \delta_{n+1} &(\text{since }\delta_{n+1}\leq \tfrac{1}{2}r_{n+1}).
\end{align*}
  Hence, $e(c,b_{n+2}) < \delta_{n+1} \leq \frac{2^{-n-1}}{27}(s-r)$.

Therefore $e(c,b) < \frac{2^{-n-1}}{27}(s-r) + d(b_{n+2},b)$ for every $n$ by induction. Since $b_n \to b$ we have that $e(c,b) = 0$, i.e., $c \sqin b$. \hfill $\square_{\text{claim}}$
\vspace{0.25em}

Finally we just need to verify that if $c \sqin b$, then $e(c,a)< s$. By the above estimate, we know that $d(b_1,b) \leq \sum_{n=1}^\infty\frac{2^{-n-1}}{3}(s-r) = \frac{1}{6}(s-r)$. We have
\begin{align*}
  d(a,b_1) = d(b_0,b_1) &\leq \max(\delta_0,r_0 + (1+\delta_0)(s_0-r_0)) \\
                        &\leq \max\left( \tfrac{1}{2}r_0,r_0 + \left(1+\tfrac{1}{108}\right)\tfrac{1}{3}(s-r) \right) \\
                        &< \max\left( \tfrac{1}{2}r_0,r_0 + \tfrac{5}{4}\cdot\tfrac{1}{3}(s-r) \right) = \tfrac{1}{4}r + \tfrac{3}{4}s.
\end{align*}
Therefore,
\[
  d(a,b) \leq d(a,b_1) + d(b_1,b) <  \tfrac{1}{4}r + \tfrac{3}{4}s + \tfrac{1}{6}(s-r) = \tfrac{1}{12}r + \tfrac{11}{12}s < s.
\]
So if $c \sqin b$, then there is an $f \sqin a$ such that $d(c,f) < s$, implying that $e(c,a) < s$, as required.
\end{proof}

\begin{prop}\label{prop:MSE-CL-char}
  Let $(M,e)$ be an $\HH$-extensional $\Lc_e$-structure with $(M,d_e)$ complete. $(M,d_e,\allowbreak{\sqin_e}) \models \MSE$ if and only if $(M,e)$ satisfies $\Lc_e$-excision.

  Furthermore, all models of $\MSE$ arise in this manner.
\end{prop}
\begin{proof}
  Let $\chi(x,\ybar,z) = \max(\min(e(x,z),-\varphi(x,\ybar)),\min(\e_\varphi-e(x,z),\varphi(x,\ybar)))$.  Suppose that $(M,d_e,\sqin_e) \models \MSE$. Fix a restricted $\Lc_e$-formula $\varphi(x,\ybar)$. Fix a tuple of parameters $\abar$. Let $b = \exc{x}{\varphi_{\sqin}(x,\abar)}{0}{\frac{1}{2}}$. Now for any $c$, we have that if $\varphi(c,\abar) \leq 0$, then $e(c,b) =0$. So $\min(e(c,b),-\varphi(c,\abar)) \leq 0$. Moreover, if $\varphi(c,\abar) \geq 1$, then $d_e(c,f) > \e_\varphi$ for all $f \sqin_e b$ (since these all satisfy $\varphi^M(f,\abar) < \frac{1}{2}$). Therefore $e(c,b) \geq \e_\varphi$. So $\min(\e_\varphi-e(c,b),\varphi(c,\abar)) \leq 0$. Since we can do this for any $c \in M$, we have that $(M,e) \models \sup_x\chi(x,\abar,b) \leq 0$, whereby $(M,e) \models \inf_z\sup_x \chi(x,\abar,z) \leq 0$. Since we can do this for any $\abar \in M$, we have that $(M,e) \models \sup_{\ybar}\inf_z\sup_x\chi(x,\ybar,z)$. Finally since this holds for any restricted $\Lc_e$-formula, we have that $(M,e)$ satisfies $\Lc_e$-excision.

  Now assume that $(M,e)$ satisfies $\Lc_e$-excision. Fix an $\Lsq$-formula $\varphi(x,\ybar)$, $\abar \in M$, and $r < s$. By passing to $r'$ and $s'$ with $r < r' < s' < s$ if necessary, we may assume that $r$ and $s$ are rational. By \cref{fact:translation}, we can fix a rational $\Lsq$-formula $\psi(x,\ybar)$ such that $\left| \psi(x,\ybar) - \frac{\varphi(x,\ybar)-r}{s-r}\right| < \frac{1}{6}$ for all $x$ and $\ybar$. Fix $\delta > 0$ with $\delta < \frac{1}{2}\e_{3\psi-1}$. Note that $\delta < \frac{1}{2}$. Apply $\Lc_e$-excision to the restricted $\Lc_e$-formula $3\psi_e(x,\abar)-1$ to get $b \in M$ such that for all $c \in M$, if $\psi_e(c,\abar) \leq \frac{1}{3}-\frac{1}{3}\delta$, then $e(c,b) < \delta$ and if $e(c,b) < \e_{3\psi-1}-\delta$, then $\psi_e(c,\abar) < \frac{2}{3}+\frac{1}{3}\delta$. Apply \cref{lem:fat-comp} to $b$ to get a set $f$ such that if $e(c,b) \leq \delta$, then $c\sqin f$ and if $c \sqin f$, then $e(c,b) \leq \frac{1}{2}\e_{3\psi-1} < \e_{3\psi-1}-\delta$.

  For any $c \in M$, suppose that $\varphi(c,\abar) \leq r$. We then have that $\psi(c,\abar) = \psi_e(c,\abar) < \frac{1}{6} < \frac{1}{3}-\frac{1}{3}\delta.$ Therefore, $e(c,b) < \delta$ and so $c \sqin f$. On the other hand, suppose that $c \sqin f$. We then have that $e(c,b) < \e_{3\psi-1}-\delta$. Therefore $\psi_e(c,\abar) < \frac{2}{3}+\frac{1}{3}\delta$, implying that $\frac{\varphi(c,\abar)-r}{s-r} < \frac{2}{3}+\frac{1}{3}\delta + \frac{1}{6} < 1$ and so $\varphi(c,\abar) < s$.

  Since we can do this for any $\varphi(x,\abar)$ and $r < s$, we have that $(M,d_e,\sqin_e)\models \MSE$.

  The `Furthermore' statement follows from the fact that if $(M,d,\sqin)\models \MSE$, then $(M,e)$ (where $e$ is defined from $\sqin$ and $d$) is $\HH$-extensional and satisfies $\Lc_e$-excision.
\end{proof}

Given \cref{prop:MSE-CL-char}, we will also use $\MSE$ to denote the $\Lc_e$-theory consisting of the $\HH$-extensionality axiom and the axiom scheme of excision.

\section{Constructing models of $\MSE$}
\label{sec:constructing-models}

In order to construct models of $\MSE$, we need to borrow techniques from the construction of models of $\GPK$. The construction also has something of the flavor of the construction of models of $\NFU$ in that it involves non-standard models of another set theory. In order to show that arbitrary metric spaces can be a set of Quine atoms\footnote{Recall that a \emph{Quine atom}, sometimes called a \emph{self-singleton}, is a set $x$ satisfying $x = \{x\}$.} in a model of $\MSE$, we will use a construction that combines elements of the tree structures in \cite{weydert_thesis} and the construction presented at the end of \cite[Sec.\ 2]{Forti1996}. The construction we give here could be generalized to allow certain other metric set structures to be embedded in models of $\MSE$, in the same vein as \cite[Sec.\ 2]{Forti1996}, but we have not pursued this here. We work in the context of $\ZF$.  %

In the following definition, $Q$ is intended to be a set of Quine atoms in our resulting model, although the models we construct here always have precisely one additional Quine atom.
\begin{defn}\label{defn:Q-basic}
  Fix a set $Q$ and a $[0,1]$-valued metric $d$ on $Q$. Assume that $Q$ does not contain any ordinal-indexed sequences. For any ordinal $\alpha$, we let $\Tc_\alpha(Q)$ be the set of all $\alpha$-sequences $x$ satisfying that
  \begin{itemize}
  \item for every $\beta<\alpha$, $x(\beta) \subseteq Q \cup \Tc_\beta(Q)$ and
  \item  for every $\beta < \gamma < \alpha$, $x(\beta)\cap Q = x(\gamma)\cap Q$ and $x(\beta)\setminus Q = \{y\res \beta : y \in x(\gamma)\setminus Q\}$.
  \end{itemize}
  Let ${\sqin} \subseteq \Tc_\alpha(Q)^2$ be a binary relation such that
  \begin{itemize}
  \item for $x \in \Tc_\alpha(Q)$ and $y \in Q$, $x\sqin y$ holds if and only if $x=y$,
  \item for $x \in Q$ and $y \sqin \Tc_\alpha(Q)\setminus Q$, $x\sqin y$ holds if and only if $x\in y(\beta)$ for every $b<\alpha$, and
  \item for $x \in \Tc_\alpha(Q) \setminus Q$ and $y \in \Tc_\alpha(Q)\setminus Q$, $x \sqin y$ if and only if $x \res \beta \in y(\beta)$ for every $\beta < \alpha$. 
 \end{itemize}
 Note that $\Tc_\alpha(Q)$ is well-defined, as $\Tc_0(Q) = Q \cup \{\varnothing\}$. 
  For any $x \in \Tc_\alpha(Q)$, we write $\tc(x)$ for the smallest subset of $\Tc_\alpha(Q)$ such that $\{y : y \sqin x\} \subseteq \tc(x)$ and if $z \sqin y \in \tc(x)$, then $z \in \tc(x)$. %
  For any $x,y \in \Tc_\alpha(Q)$, we define
  \begin{align*}
    \rho_0^{Q,\alpha}(x,y) &\coloneqq d_{\HH}(\tc(x)\cap Q,\tc(y)\cap Q), \\ 
    e_{\beta}^{Q,\alpha}(z,y) & \coloneqq \inf_{w \sqin y}\rho_\beta^{Q,\alpha}(z,w), \\
    \rho_{\beta+1}^{Q,\alpha}(x,y) &\coloneqq \max\left( \sup_{z \sqin x}e_\beta^{Q,\alpha}(z,y)  ,\sup_{w \sqin y} e_\beta^{Q,\alpha}(w,x)\right), \\ %
    \rho_{\lambda}^{Q,\alpha}(x,y) &\coloneqq \sup_{\beta<\lambda}\rho_{\beta}^{Q,\alpha}(x,y), 
  \end{align*}
  for all $\beta$ and $\lambda$ a limit ordinal, where $\sup \varnothing = 0$ and $\inf \varnothing = 1$.
\end{defn}
We will often suppress the superscript $^{Q,\alpha}$. Since the supremum of a family of pseudo-metrics is always a pseudo-metric, an easy inductive argument shows that $\rho_\beta$ is a pseudo-metric for every $\beta \in \Ord\cup\{\infty\}$. It is also immediate that for any $x \in Q$, $\tc(x) = \{x\}$, and so for $x,y \in Q$, $\rho_\beta(x,y) = d(x,y)$ for every~$\beta$. Finally, it can be shown that if $x(\gamma) = y(\gamma)$ for all $\gamma \leq \beta$, then $\rho_\beta(x,y) = 0$.

Also, while we will not need it, we should note that $\Tc_\alpha(\varnothing)$ is precisely the tree structure of height $\alpha$ of \cite{weydert_thesis} and in this case, $\rho_\beta(x,y)$ is $0$ if and only if $x\res \beta = y \res \beta$ and is $1$ otherwise. $\rho_\beta$ is of course also closely related to the $\sim_\beta$ relation of \cite{malitz_1976}. 

\begin{lem}\label{lem:rho-increasing}
  For any $a,b \in \Tc_\alpha$, $\beta \mapsto \rho_\beta(a,b)$ and $\beta \mapsto e_\beta(a,b)$ are both non-decreasing functions of $\beta$.
\end{lem}
\begin{proof}
  Proceed by induction on $\beta$. Limit stages are obvious, so assume that we know that $\gamma \mapsto \rho_\gamma(a,b)$ and $\gamma \mapsto e_\gamma(a,b)$ are increasing functions for any $a,b \in \Tc_\alpha$ on the interval $[0,\beta]$ and consider $\rho_{\beta+1}(x,y)$.
  
  If $\beta = 0$, then we just need to argue that $\rho_1(a,b) \geq \rho_0(a,b) = d_{\HH}(\tc(a) \cap Q, \tc(b) \cap Q)$. Suppose that $\rho_0(a,b) > r$. Without loss of generality, this implies that there is a $c \in \tc(a)\cap Q$ such that $\inf\{d(c,z) : z \in \tc(b) \cap Q\} > r$. Since $c \in \tc(a)$ and $c \sqin c$, there is an $f \sqin a$ such that $c \in \tc(f)$. Since $c \in \tc(f)\cap Q$ and since $\tc(g) \subseteq \tc(b)$ for any $g \sqin b$, we have that $\rho_0(f,g) > r$ for any $g \sqin b$. Therefore $\rho_1(a,b) \geq r$. Since we can do this for any $r$, we have that $\rho_1(a,b) \geq \rho_0(a,b)$.

  If $\beta > 0$, then for any $\gamma \leq \beta$, we have that $e_\gamma(u,v) \leq e_\beta(u,v)$ by the induction hypothesis, so
  \[
    \max\left( \sup_{z \sqin a}e_\gamma (z,b)  ,\sup_{w \sqin b} e_\gamma (w,a)\right) \leq \max\left( \sup_{z \sqin a}e_\beta(z,b)  ,\sup_{w \sqin b} e_\beta(w,a)\right)
  \]
  and therefore $\rho_\gamma(a,b) \leq \rho_{\gamma+1}(a,b) \leq \rho_{\beta+1}(a,b)$, as required. The fact that $e_{\beta+1}(a,b) \geq e_\beta(a,b)$ is immediate.
\end{proof}

\begin{lem}\label{lem:V-alpha-disc-embedding}
 For any $(Q,d)$ and ordinals $\alpha<\beta$, there is a unique $v_\alpha \in \Tc_\beta(Q)$ such that $(v_\alpha,\sqin)$ and $(V_\alpha,\in)$ are isomorphic and for any $\gamma \in (\alpha,\beta)$ and distinct $a,b \sqin v_\alpha$, $\rho_\gamma(a,b) = 1$.
\end{lem}
\begin{proof}
  Fix an ordinal $\beta$. We will prove this for all $\alpha < \beta$ by induction. For $V_0 = \varnothing$, the statement is witnessed by the sequence $v_0(\gamma) = \varnothing$ in $\Tc_\beta(Q)$.

  Now assume that for some $\alpha < \beta$, the statement is known for all $\delta < \alpha$. If $\alpha$ is a successor and equal to $\gamma + 1$, let $v_\alpha$ be defined by $v_\alpha(0) = \{\varnothing\}$, $v_\alpha(\sigma+1) = \Pc(v_\gamma(\sigma))$, and $v_\alpha(\lambda) = \{x : x~\text{ is a}~\lambda~\text{sequence,}~(\forall \sigma < \lambda)x\res \sigma \in v_\alpha(\sigma)\}$ for any limit ordinal $\lambda$. Since the statement holds for $\gamma$, we have that $\rho_\gamma$ is $\{0,1\}$-valued on the $\sqin$-elements of $v_\gamma$, we get that $\rho_\alpha=\rho_{\gamma+1}$ is $\{0,1\}$-valued on the $\sqin$-elements of $v_\alpha$. Furthermore, since $(v_\gamma,\sqin)$ is isomorphic to $(V_\gamma,\in)$, it follows immediately that $(v_\alpha,\sqin)$ is isomorphic to $(V_\alpha,\in)$. 

  If $\alpha$ is a limit, then let $v_\alpha(\sigma) = \bigcup_{\gamma < \alpha}v_\gamma(\sigma)$ for every $\sigma < \beta$. The required statements are obvious.
\end{proof}

\begin{defn}\label{defn:Q-topology}
  For any $Q$ and $\alpha$, we let $\tau_{Q,\alpha}$ be the topology on $\Tc_\alpha(Q)$ generated by sets of the form $\{y \in \Tc_\alpha(Q) : \rho_\beta(x,y) < \e\}$ for $x \in \Tc_\alpha(Q)$, $\beta<\alpha$, and $\e > 0$.
\end{defn}

It is immediate from basic topological facts that for any $X \subseteq \Tc_\alpha(Q)$, there is a unique smallest closed set $\overline{X}$ containing $X$. More importantly, we have the following.

\begin{prop}\label{prop:closures-are-elements}
  For any $Q$, limit $\alpha$, and closed $F \subseteq \Tc_\alpha(Q)$, there is an $x \in \Tc_\alpha(Q)$ such that $F = \{y \in \Tc_\alpha(Q): y \sqin x\}$. 
\end{prop}
\begin{proof}
  For any $\beta<\alpha$, let $x(\beta) \coloneqq \{y\in \Tc_\beta(Q) : y~\text{extends to an element of}~F\}$. $x$ is clearly an element of $\Tc_\alpha(Q)$. Furthermore, we clearly have that if $y \in F$, then $y \sqin x$. So now we just need to show the converse.

  Suppose that $y \sqin x$. We would like to show that $y$ is in the closure of $F$ and therefore is in $F$. In order to do this, it is sufficient to show that $\inf\{\rho_\beta(y,z) : z \in F\} = 0$ for each $\beta < F$. For each $\beta < \alpha$, find $z \in F$ such that $y(\beta+1)$ is an initial segment of $z$. We now have that $\rho_\beta(y,z) = 0$. Since we can do this for any $\beta < \alpha$, we have that $y$ is in the closure of $F$.
\end{proof}

What will ultimately be relevant to us is that the above facts are first-order properties of the structure $(V_{\alpha+\omega},\alpha,Q,d,\sqin,\Tc_\alpha(Q))$ (assuming $\Tc_\alpha(Q)$ is an element of $V_{\alpha+\omega}$). This is part of the motivation for \cref{defn:non-standard-gauges}.

We will also need the following.  

\begin{lem}\label{lem:Q-embedding}
  Fix a metric space $(Q,d)$ and a limit ordinal $\alpha$. Let $\ol{Q}$ be the $\tau_{Q,\alpha}$-closure of $Q \subset \Tc_\alpha(Q)$. For any $z \in \ol{Q}$, there is an $x \in Q$ such that $\rho_\beta(x,z) = 0$ for all $\beta < \alpha$. 
\end{lem}
\begin{proof}
  First we need to show that if $x \in \ol{Q}$, then $|\tc(x)\cap Q| = 1$. Suppose that $\tc(x) \cap Q$ has more than one element. Let $y$ and $z$ be distinct elements of $\tc(x) \cap Q$. Suppose that $d(y,z) > r$. We now immediately have that $\rho_0(x,w) > \frac{1}{2}r$ for any $w \in Q$. Therefore $x \notin \ol{Q}$. On the other hand, suppose that $\tc(x) \cap Q = \varnothing$. Then likewise, $\rho_0(x,w) = 1$ for any $w \in Q$. Therefore $x \notin \ol{Q}$.

  Now we need to argue that if $x \in \ol{Q}$, then for any $y \sqin x$, $y \in \ol{Q}$ as well. Suppose $y \sqin x$ and $y \notin \ol{Q}$. By definition, this implies that there is a $\beta < \alpha$ such that $\inf\{\rho_\beta(y,z) : z \sqin x \} = r> 0$, but this implies that $\rho_{\beta+1}(x,w) \geq r $ for all $w \in Q$ and so $x \notin \ol{Q}$.
  
  For any $x \in \ol{Q}$, let $\pi(x)$ denote the unique element of $Q$ that is in $\tc(x)$. We need to show that $\rho_\beta(x,\pi(x)) = 0$ for all $\beta < \alpha$. Clearly $\rho_0(y,\pi(x)) = 0$ for any $y \in \ol{Q}$ with $\pi(y) = \pi(x)$. Suppose that $\rho_\gamma(y,\pi(x)) = 0$ for all $\gamma < \beta$ and $y \in \ol{Q}$ with $\pi(y) = \pi(x)$. If $\beta$ is a limit, then $\rho_\beta(x,\pi(x)) = 0$. Assume that $\beta = \delta+1$ for some $\delta$. Fix $y$ with $\pi(y) = \pi(x)$. Fix $z \sqin y$. We clearly have that $\tc(z) \cap Q \subseteq \tc(y) \cap Q$. It also must be the case that $z \in \ol{Q}$. Therefore we must have that $\pi(z) = \pi(y) = \pi(x)$ as well, so by the induction hypothesis, we have that $\rho_\delta(z,\pi(x)) = 0$. Since we can do this for any $z \sqin y$, we have that $\rho_{\delta+1}(y,\pi(x)) = 0$, as required.
\end{proof}

\begin{defn}\label{defn:non-standard-gauges}
  Fix a tuple $(Q,d,\sqin)$ as in \cref{defn:Q-basic} and an infinite ordinal $\alpha$ such that $\Tc_\alpha(Q)$ is an element of $V_{\alpha+\omega}$. We will assume that restricted $\Lc_e$-formulas are elements of $V_{\alpha+\omega}$.

  Let $(M,\alpha^M,Q^M,d^M,\sqin^M,\Tc^M)$ be a structure elementarily equivalent to $(V_{\alpha+\omega},\alpha,Q,d,\allowbreak\sqin,\allowbreak\Tc_\alpha(Q))$. We write $\rho^M_\beta(x,y)$ and $e^M_\beta(x,y)$ for the functions in $M$ given by \cref{defn:Q-basic} computed internally.

  Given any $r \in \Rb^M$ satisfying $|r| \leq n$ for some standard natural $n$, the \emph{standard part of $r$}, written $\st(r)$, is the unique standard real satisfying $r \geq t$ if and only if $\st(r) \geq t$ for all standard rationals $t$.
  
  A \emph{gauge on $M$} is a non-increasing function $s: \alpha^M \to [0,1]$ (where $[0,1]$ is the standard unit interval) with $s(0) = 1$. An \emph{internal gauge on $M$} is a non-increasing function $s \in M$ from $\alpha^M$ to $[0,1]^M$ with $s(0) = 1$. An internal gauge on $M$ is \emph{$\e$-smooth} if
  \begin{itemize}
  \item $s(0) = s(1)$,
  \item $s(\beta) = 0$ for all sufficiently large $\beta \in \alpha^M$,
  \item for every $\beta \in \alpha^M$, $s(\beta) < s(\beta+1) + \e$, and
  \item for any limit $\lambda \in \alpha^M$, there is a $\beta < \lambda$ such that $s(\beta) = s(\lambda)$.
  \end{itemize}
  Given an internal gauge $s$ on $M$, the \emph{standard part of $s$}, written $s^{\st}$, is $\st \circ s$.

  Given a gauge $s$ on $M$, we define the functions
  \[
    \rho_s(x,y) \coloneqq \sup_{\beta \in \alpha^M}\min(\rho^M_\beta(x,y),s(\beta))
  \]
  and $e_s(x,y) = \inf_{w \sqin^M y}\rho_s(x,w)$. For any internal gauge $s$ we write $\rho_s$ and $e_s$ for the corresponding quantities computed internally in $M$ and we write $\rho_s^{\st}$ and $e_s^{\st}$ for their corresponding standard parts.

  Given two gauges $s_0$ and $s_1$ on $M$, we write $\lVert s_0-s_1 \rVert$ for the quantity $\sup_{\beta \in \alpha^M}|s_0(\beta) - s_1(\beta)|$.
\end{defn}

Note that $\rho^{\st}_s = \rho_{s^{\st}}$ and $e^{\st}_s = e_{s^{\st}}$ for any $M$ as in \cref{defn:non-standard-gauges}. Since $\rho_s(x,y)$ is the supremum of a family of pseudo-metrics, it is itself a pseudo-metric. Finally, it is trivial that for any gauge $s$ on $M$, $(\Tc^M,e_s)$ is an $\Lc_e$-structure.

In the following, we will write `($M$ as in \ref{defn:non-standard-gauges}.)' to mean that the structure $M = (\Tc^M,\alpha^M,\Rb^M)$ satisfies the conditions in \cref{defn:non-standard-gauges}.

\begin{lem}\label{lem:gauge-continuity}
  ($M$ as in \ref{defn:non-standard-gauges}.) Fix a restricted $\Lc_e$-formula $\varphi(\xbar)$ and a tuple $\abar \in \Tc^M$.
  \begin{enumerate}
  \item For any gauges $s$ and $t$ on $M$,
    \[
      |\varphi^{(\Tc^M,e_{s})}(\abar) - \varphi^{(\Tc^M,e_t)}(\abar)| \leq v(\varphi)\lVert s-t \rVert.
    \]
  \item For any internal gauge $s$ on $M$,
    \[
      \st((\varphi^{(\Tc^M,e_s)}(\abar))^M) = \varphi^{(\Tc^M,e_s^{\st})}(\abar),
    \]
    where $(\varphi^{(\Tc^M,e_s)}(\abar))^M$ is the value of $\varphi^{(\Tc^M,e_s)}(\abar)$ computed internally in $M$.
  \end{enumerate}
\end{lem}
\begin{proof}
  It is straightforward to show that for any $a,b \in \Tc^M$, $|\rho_s(a,b) - \rho_t(a,b)| \leq \lVert s-t \rVert$. This implies likewise that for any $a,b \in \Tc^M$, $|e_s(a,b) - e_t(a,b)| \leq \lVert s-t \rVert$. From this, 1 follows by an induction argument.  2 also follows from an easy induction argument.
\end{proof}

\begin{lem}\label{lem:rho-Haus-alt-def}
  ($M$ as in \ref{defn:non-standard-gauges}.) For any $\beta \in \alpha^M$,
  \[
    \rho_{\beta+1}^M(x,y) = \sup_{z \in \Tc^M}|e_\beta^M(z,x)-e_\beta^M(z,y)|.
  \]
\end{lem}
\begin{proof}
  This follows immediately from the fact that $\rho_\beta^M$ is a pseudo-metric on $\Tc^M$ and $\rho_{\beta+1}^M(x,y)$ is precisely the Hausdorff distance between $\{z : z \sqin^M x\}$ and $\{z : z \sqin^M y\}$ with respect to $\rho_{\beta+1}^M$.
\end{proof}

\begin{lem}\label{lem:approx-H-ext}
  ($M$ as in \ref{defn:non-standard-gauges}.) Fix $\e \in (0,1]^M$ and an $\e$-smooth internal gauge $s$ on $M$. Let $d_{e,s}(x,y) \coloneqq \sup_{z}|e^M_s(z,x)-e^M_s(z,y)|$. The following statements hold internally in $M$:
  \begin{enumerate}
  \item $\rho_s(a,b) \leq d_{e,s}(a,b) \leq \rho_s(a,b) + \e$ for all $a,b \in \Tc^M$.
  \item $|e_s(a,b) - \inf_z \min(d_{e,s}(a,z) + 2e_s(z,b),1)| \leq \e$ for all $a,b \in \Tc^M$.
  \end{enumerate}
\end{lem}
\begin{proof}
  First note that by definition, $d_{e,s}(x,y)$ is the Hausdorff pseudo-metric induced by the pseudo-metric $\rho_s(x,y)$. So in particular we also have that
  \[
    d_{e,s}(x,y) = \max\left(\sup_{z \sqin x}e_s(z,y),\sup_{w \sqin y}e_s(w,x)\right).
  \]

  For 1, fix $a$ and $b$ in $\Tc$. If $a$ and $b$ are both $\sqin$-empty, then $d_{e,s}(a,b) = \rho_s(a,b) = 0$. If one of them is $\sqin$-empty and the other isn't, then $d_{e,s}(a,b) = \rho_s(a,b) = 1$. So assume that they are both non-$\sqin$-empty.
  
  Suppose that $\rho_s(a,b) > r$. This implies that there is a $\beta \in \alpha^M$ such that $\min(\rho_\beta(a,b),s(\beta)) > r$, which implies that $\rho_\beta(a,b) > r$. If $\beta$ is a limit ordinal, then $\rho_\beta(a,b) = \sup_{\gamma < \beta}\rho_\gamma(a,b)$, so, since $s$ is non-increasing, we may assume that $\beta$ is not a limit ordinal. Since $s(0) = s(1)$, we may assume that $\beta > 0$ by \cref{lem:rho-increasing}. So let $\gamma + 1 = \beta$. We may now assume without loss of generality that there is a $c \sqin a$ such that $e_\gamma(c,b) > r$, implying that $\rho_\gamma(c,f) > r$ for all $f \sqin b$. Therefore we have that $\rho_s(c,f) \geq \min(\rho_s(c,f),s(\gamma)) > r$ for all $f \sqin b$, whence $e_s(c,b) \geq r$ and $d_{e,s}(a,b) \geq r$. Since we can do this for any $r < \rho_s(a,b)$, we have that $d_{e,s}(a,b) \geq \rho_s(a,b)$.

  Now suppose that $d_{e,s}(a,b) > r$ for some $r > 0$. We may assume without loss of generality that there is a $c \sqin a$ such that $e_s(c,b) > r$. So in particular, $\rho_s(c,f) > r$ for all $f \sqin b$. Therefore, for any such $f$, there is a $\beta_f \in \alpha^M$ such that $\min(\rho_{\beta_f}(c,f),s(\beta_f)) > r$. Since $s$ is $\e$-smooth, there is a largest $\gamma \in \alpha^M$ such that $s(\gamma) > r$. Note that we must have $\gamma \geq \beta_f$ for all $f \sqin b$. So now we actually know that $\rho_\gamma(c,f) > r$ for all $f \sqin b$. Therefore $e_\gamma(c,b) \geq r$ and so $\rho_{\gamma+1}(a,b) \geq r$, whence $\rho_s(a,b) \geq \min(\rho_{\gamma+1}(a,b),s(\gamma+1)) > r - \e$. Since we can this for any $r < d_{e,s}(a,b)$, we have that $\rho_s(a,b) \geq d_{e,s}(a,b)-\e$, as required.

  For 2, it follows from the $\Leftarrow$ direction of the proof of \cref{lem:ax-H-ext-char} that $e_s(a,b) = \inf_z \min(\rho_s(a,z) + 2e_2(z,b), 1)$ for all $a,b \in \Tc^M$. It is immediate from part 1 that
  \[
    |\inf_z \min(d_{e,s}(a,z) + 2e_2(z,b), 1) - \inf_z \min(\rho_s(a,z) + 2e_2(z,b), 1)| \leq \e
  \]
  for all $a,b \in \Tc^M$, so the required result follows.
\end{proof}

\begin{lem}\label{lem:approx-Lipschitz}
  ($M$ as in \ref{defn:non-standard-gauges}.) Fix $\e \in (0,1]^M$ and let $s$ be an $\e$-smooth internal gauge on $M$. For any\footnote{Possibly non-standard, although we do not need this.} $\Lc_e$-formula $\varphi(\abar,\bbar)$ and any $\abar,\bbar \in \Tc^M$,
  \[
    |\varphi^{(\Tc^M,e_s)}(\abar) - \varphi^{(\Tc^M,e_s)}(\bbar)| \leq 2v(\varphi)d_{e,s}(\abar,\bbar),
  \]
  where $d_{e,s}(\abar,\bbar) = \max_{i<|\abar|}d_{e,s}(a_i,b_i)$.
\end{lem}
\begin{proof}
  We prove this by induction on formulas. If $\varphi$ is $e(x,y)$, then we have
  \begin{align*}
    |e(a_0,a_1)-e(b_0,b_1)| &\leq |e(a_0,a_1)-e(b_0,a_1)| + |e(b_0,a_1)-e(b_0,b_1)| \\
                            &\leq  \rho_s(a_0,b_0)+ d_{e,s}(a_1,b_1) \\
                            &\leq d_{e,s}(a_0,b_0)+ d_{e,s}(a_1,b_1) \\
                            &\leq 2d_{e,s}(a_0a_1,b_0b_1)
  \end{align*}
  by \cref{lem:approx-H-ext}. The argument for connectives and quantifiers the same as in \cref{lem:formula-v}.
\end{proof}

\begin{lem} \label{lem:approx-excision}
  ($M$ as in \ref{defn:non-standard-gauges}.) Fix $\e \in (0,1]^M$ and let $s$ be an $\e$-smooth internal gauge on $M$. For any $\Lc_e$-formula $\varphi(x,\ybar)$, $M$ satisfies that
  \[
    (\Tc^M,e_s) \models   \sup_{\ybar}\inf_z \sup_x \max(\min(e(x,z),-\varphi(x,\ybar)),\min(\e_\varphi-e(x,z),\varphi(x,\ybar) - 1)) \leq 2v(\varphi)\e.
  \]
\end{lem}
\begin{proof}
  For any $\Lc_e$-formula $\varphi(x,\ybar)$ and any $\abar \in M$, let $B_0 = \{ x \in \Tc^M : \varphi^{(\Tc^M,e_s)}(x,\abar) \leq 0\}$. We need to argue that $B_0$ is closed in the topology on $\Tc^M$ given in \cref{defn:Q-topology}. Suppose that $c \notin B_0$. This means that $\varphi^{(\Tc^M,e_s)}(c,\abar) > 0$. Since $e_s$ is $2$-Lipschitz with regards to $\rho_s$, this implies that there is a $\delta > 0$ such that for any $c' \in \Tc^M$ with $\rho_s(c,c') < \delta$, $\varphi^{(\Tc^M,e_s)}(c',\abar) > 0$ as well. Since $s$ is $\e$-smooth, there is a $\gamma \in \alpha^M$ such that $s(\gamma) = 0$. By \cref{lem:rho-increasing}, we know that if $\rho_\gamma(c,c') < \frac{1}{2} \delta$, then $\rho_s(c,c') < \delta$. Therefore we have that the set $\{x : \rho_{\gamma}(x,c) < \frac{1}{2}\delta\}$ is disjoint from $B_0$. Since we can do this for any $c \notin B_0$, we have that $B_0$ is closed.

  Let $b$ be the unique element of $\Tc^M$ coextensive with $B_0$ (which exists by \cref{prop:closures-are-elements}). For any $c \in \Tc^M$, if $\varphi^{(\Tc^M,e_s)}(c,\abar) \leq 0$, then $c \sqin b$ by our choice of $c$ and so $e_s(c,b) = 0$. Therefore $\min(e_s(c,b),-\varphi(c,\abar)) \leq 0$ for all $c \in \Tc^M$. On the other hand, if $e_s(c,b) \leq \e_\varphi$, then there is an $f \sqin b$ such that $\rho_s(c,f) < \e_\varphi + \sigma$ for any $\sigma > 0$. Therefore, $\varphi(c,\abar) < 2v(\varphi)d_{e,s}(c,f) < 2v(\varphi)(\e_\varphi+\sigma+\e)$ by Lemmas~\ref{lem:approx-H-ext} and \ref{lem:approx-Lipschitz}. Since $2v(\varphi)\e_\varphi < 1$ and since we can do this for any $\sigma > 0$, we have that $\varphi(c,\abar) \leq 1 + 2v(\varphi)\e$. Therefore $\min(\e_\varphi-e(c,b),\varphi(c,\abar) - 1) \leq 2v(\varphi)\e$ for any $c \in \Tc^M$. 
\end{proof}

\begin{lem}\label{lem:e-smooth-gauges-exist}
  ($M$ as in \ref{defn:non-standard-gauges}.) For any (external) gauge $s$ on $M$ with dense image in $[0,1]$ and (standard) rational $\e \in (0,1]$, there is an $\e$-smooth internal gauge $t$ such that $\lVert s-t^{\st} \rVert \leq \e$.
\end{lem}
\begin{proof}
  Find standard $n$ large enough that $\frac{1}{n} < \frac{1}{2}\e$. For each $i<n$, find $\beta_i \in \alpha^M$ such that $\frac{i}{n} \leq s(\beta_i) < \frac{i+1}{n}$. Note that since the range of $s$ is dense, none of the $\beta_i$'s are $0$ or $1$. Also note that $(\beta_i)_{i<n}$ is a decreasing sequence of ordinals. For any $\gamma \in \alpha^M$, let
  \[
    t(\gamma) =
    \begin{cases}
      0 & \gamma \geq \beta_0 \\
      \frac{i}{n} & 0<i< n,~\beta_{i-1} > \gamma \geq \beta_i \\
      1 & \beta_{n-1} > \gamma
    \end{cases}.
  \]
Clearly $t(\gamma) = 0$ for all sufficiently large $\gamma$. We also have that $s(\gamma) < s(\gamma+1) + \frac{2}{n} < s(\gamma+1) + \e$ for all $\gamma \in \alpha^M$. Finally, the limit ordinal condition in the definition of $\e$-smooth is clearly met, so $t$ is $\e$-smooth.

Now for any $\gamma$, we have that if $t(\gamma) = 0$, then $\gamma \geq \beta_0$ and so $s(\gamma) \leq \frac{1}{n} < \e$. If $t(\gamma) \in (0,1)$, then there is a positive $i<n$ such that $\beta_{i-1} > \gamma \geq \beta_i$, implying that $\frac{i-1}{n} \leq s(\gamma) \leq \frac{i+1}{n}$, so $|s(\gamma) - t(\gamma)| = |s(\gamma) - \frac{i}{n}| < \e$. And if $t(\gamma) = 1$, then $s(\gamma) \geq s(\beta_{n-1}) \geq \frac{n-1}{n}$, so $|s(\gamma)-t(\gamma)| \leq \frac{1}{n} < \e$. Therefore $\lVert  s-t \rVert \leq \e$, as required.
\end{proof}

In order to proceed we will need a fact from model theory. This is similar to the approach typically used to build partially standard models of $\NFU$.

\begin{lem}\label{lem:partially-standard}
  For any ordinal $\sigma$, there is an $\alpha > \sigma$ and a structure $(M,\alpha^M) \equiv (V_{\alpha+\omega},\alpha)$ such that $V_\sigma^M$ is isomorphic to $V_\sigma$ and there is a set of $M$-ordinals less than $\alpha^M$ that is order-isomorphic to $\Qb$.
\end{lem}
\begin{proof}
  Let $\kappa = |V_{\sigma}|$. Let $\alpha = \beth_{(2^\kappa)^+}$. Expand $(V_{\alpha +\omega},\in)$ by Skolem functions. By \cite[Lem.\ 7.2.12]{tent_ziegler_2012}, we can find a $V_{\sigma}$-indiscernible sequence $I$ with order type $\Qb$ in some elementary extension $N$ of $(V_{\alpha+\omega},\in,\alpha,\text{Skolem functions})$ such that for any increasing sequence $a_0<\dots < a_{n-1}$ in $I$, there are ordinals $\delta_0,\dots,\delta_{n-1}  < \alpha$ with $\tp(\abar / V_{\sigma}) = \tp(\bar{\delta}/V_{\sigma})$. Let $M$ be the Skolem hull of $V_{\sigma} \cup I$. For any element $a$ of $V_\sigma^M$, there is a Skolem function $f$ and tuples $\bbar \in V_\sigma$ and $\cbar \in I$ such that $a = f(\bbar,\cbar)$. By construction, there is a tuple $\bar\delta \in V_\alpha$ such that $\tp(\cbar/V_{\sigma}) = \tp(\bar\delta/V_{\sigma})$. Since $f(\bbar,\bar\delta) \in V_{\sigma}$, we have that $f(\bbar,\cbar) = f(\bbar,\bar\delta)$. Therefore $V_\sigma^M = V_\sigma$, as required.
\end{proof}

\begin{thm}\label{thm:MSE-consistent}
  For any complete metric space $(Q,d)$ with $[0,1]$-valued metric and any ordinal $\sigma$, there is a model $N$ of $\MSE$ such that $N$ contains a set of Quine atoms isometric to $(Q,d)$ and a $1$-discrete set $v$ such that $(v,\sqin_e)$ is isomorphic to $(V_\sigma,\in)$ (and therefore $\Ord^N$ has standard part of length at least $\sigma$ and if $\sigma$ is infinite, $N \models \Inf(v)$). In particular, $\MSE$ is consistent.
\end{thm}
\begin{proof}
  Fix $(Q,d)$ as in \cref{defn:Q-basic}. Fix some ordinal $\sigma$. We may assume that $(Q,d) \in V_\sigma$. Apply \cref{lem:partially-standard} to get a structure $M$ elementarily equivalent to $V_{\alpha+\omega}$ for some $\alpha > \sigma$ such that the standard part of $M$ contains $V_\sigma$.  Let $J$ be a set of $M$-ordinals less than $\alpha^M$ order-isomorphic to $\Qb$. Since $\Tc_\alpha(Q)$ is definable from $Q$ and $\alpha$, there is an element $\Tc^M$ in $M$ realizing the same type over $Q$ and $\alpha^M$. In this way we can regard $M$ as a structure satisfying the conditions of \cref{defn:non-standard-gauges}. 

  $J$ is also order-isomorphic to $(0,1)\cap \mathbb{Q}$. Let $f$ be an order isomorphism witnessing this. Define $s : \Ord^M \to [0,1]$ by $s(\beta) = \inf\{f(\gamma) : \gamma \in J,~\gamma \leq \beta\}$ with $\inf \varnothing = 1$. This is clearly a gauge on $M$. Let $N = (\Tc^M,e_s)$.

  We need to show that for any axiom $\varphi$ of $\MSE$ (i.e., those listed in Definitions~\ref{defn:H-ext-ax} and \ref{defn:ax-sch-exc}) and any $\e > 0$,  $N \models \varphi \leq \e$. If $\varphi$ is the $\HH$-extensionality axiom, we can find with \cref{lem:e-smooth-gauges-exist} a $\frac{1}{2}\e$-smooth internal gauge $t$ such that $\lVert s-t^{st} \rVert \leq \frac{1}{2}\e$. By Lemmas~\ref{lem:gauge-continuity} and \ref{lem:approx-H-ext}, we have that $\varphi^{N} \leq \varphi^{(\Tc^M,e_t)}+\frac{1}{2}\e \leq \e$. If $\varphi$ is the excision axiom for the formula $\psi$, then we can do the same with a $\frac{1}{2v(\psi)}\e$-smooth internal gauge $t$ by \cref{lem:approx-excision}. Since we can do this for any $\varphi \in \MSE$ and $\e > 0$, we have that $N \models \MSE$.

  Finally we just need to verify that the set of Quine atoms isomorphic to $(Q,d)$ and the set isomorphic to $V_\sigma$ exist in $N$. Let $q$ be the element of $\Tc^M$ coextensive with the set $Q^\ast$ defined in \cref{lem:Q-embedding}. We have that $Q$ is a dense subset of $q$ and furthermore $\rho_\beta$ agrees with $d$ on $Q$ for all $\beta < \alpha$, therefore $(q,\rho_s)$ is isometric to $(Q,d)$, since $Q$ is metrically complete.

  Finally, let $v = v_\sigma$ as defined in the proof of \cref{lem:V-alpha-disc-embedding}. Since $\sigma(\sigma) = 1$, we have that $v$ is $1$-discrete. The relation $x\sqin_e y$ (i.e., $e(x,y) = 0$) agrees with $x\sqin y$ for $\sqin$-elements of $v$, so we have that $(v,\sqin_e)$ is isomorphic to $(V_\sigma,\in)$.
\end{proof}

One thing to note with regards to \cref{thm:MSE-consistent} is that if $M$ satisfies the axiom of choice, then the resulting structure $N$ will satisfy the axiom of choice in all of its uniformly discrete sets. Conversely, if there is a set $x \in V_\alpha^M$ witnessing the failure of the axiom of choice, then the axiom of choice will fail for the corresponding $1$-discrete set in $N$. Since we did not use the axiom of choice at any point in our construction, this establishes that choice for uniformly discrete sets is independent of $\MSE$.

Recall that an $\Lc_e$-structure $M$ is \emph{pseudo-finite} if for every restricted $\Lc_e$-sentence $\varphi$ and $r$, if $M \models \varphi < r$, then there is a finite $\Lc_e$-structure $N$ such that $N \models \varphi < r$. %

\begin{thm}\label{thm:MSE-pseudo-finite}
  There is a pseudo-finite model of $\MSE$.
\end{thm}
\begin{proof}
  Consider the structure $M = (V_{\omega+\omega},\omega,\varnothing,\dots)$. For each $n \in \Nb$, let $s_n$ be the scale on $M$ defined by $s_n(i) = \min(\max(1-\frac{i-1}{n},0),1)$. This is easily seen to be $\frac{1}{n}$-smooth. The quotient of $\Tc_\omega(\varnothing)$ by the pseudo-metric $\rho_{s_n}$ is finite, so any ultraproduct of the sequence $(\Tc_\omega(\varnothing)/\rho_{s_n},e_{s_n})_{n \in \Nb}$ is a pseudo-finite model of $\MSE$ by Lemmas~\ref{lem:approx-H-ext} and \ref{lem:approx-excision}.
\end{proof}

It is straightforward to show that no pseudo-finite model of $\MSE$ can satisfy $\Inf$, as no pseudo-finite structure can interpret Robinson arithmetic.

\section{Translation to \L ukasiewicz logic}
\label{sec:formalizing-in-LL}

It was observed in \cite{Caicedo2014} that there is a strong connection between continuous logic and \L ukasiewicz-Pavelka predicate logic. This logic extends \L ukasiewicz logic with $0$-ary connectives for each rational $r \in [0,1]$. Since every unary connective $f(x)$ in rational Pavelka logic is piecewise linear and has that $\frac{\mathrm{d}}{\mathrm{d}x}f(x) \in \Zb$ for all but finitely many $x \in [0,1]$, it is immediate that $\frac{x}{2}$ is not a connective that can be formed in it. One might think that this would prevent rational Pavelka logic from being logically complete in the sense of continuous logic, but as pointed out in \cite[Prop.~1.17]{Caicedo2014}, the connective $x \mapsto \frac{1}{2}x$ is a uniform limit of connectives in rational Pavelka logic:
\[
\lim_{n\to \infty}\max_{1\leq i \leq n}\min\left( \frac{i}{n},x \dotdiv \frac{i}{n} \right) = \frac{1}{2}x
\]
uniformly for all $x \in [0,1]$. This implies that all $[0,1]$-valued formulas in continuous logic are uniform limits of \L ukasiewicz-Pavelka formulas.

We would like to take the opportunity to observe that more than this is true. On the level of \emph{conditions} rather than \emph{formulas}, ordinary \L ukasiewicz logic is already logically complete relative to continuous logic in the following sense: For every restricted formula $\varphi(\xbar)$ (in either the sense of \cite[Sec.~3]{MTFMS} or the more permissive sense of \cite[Sec.~1.3]{HansonThesis}), there is a formula $\psi(\xbar)$ using only the connectives $1$ and $x\dotdiv y$ such that for any metric structure $M$ and tuple $\abar \in M$, $\varphi^M(\abar) \leq 0$ if and only if $\psi^M(\abar) \leq 0$.\footnote{Note however that not all continuous functions form $[0,1]^n \to [0,1]$ are uniform limits of expressions in propositional \L ukasiewicz logic, as any such expression maps $\{0,1\}^n$ to $\{0,1\}$, so \L ukasiewicz-Pavelka logic is strictly stronger in a sense that matters to continuous logic.} As a consequence of this, any continuous first-order theory or type (in a language with $[0,1]$-valued predicates) can be axiomatized entirely in \L ukasiewicz logic. This is likely obvious to those who are well-versed in \L ukasiewicz logic and its extensions, but we think it is worthwhile to write out explicitly. This fact is a consequence of results in \cite{Hjek2000}, but given the amount of translation needed to apply these results, we will sketch an argument here.

Say that a restricted formula $\varphi$ is a \emph{rational affine literal} if it is a rational affine combination of atomic formulas. Say that a quantifier-free formula is in \emph{maximal affine normal form} or \emph{max ANF} if it is $\max_{n<N}\min_{m<M_n}\varphi_{nm}$ where each $\varphi_{nm}$ is a rational affine literal. Say that a formula $\varphi$ is in \emph{prenex max ANF} if is a string of quantifiers followed by a max ANF formula. It is not too hard to show (and is written out explicitly in \cite[Prop.~1.4.12]{HansonThesis}) that every restricted formula is equivalent to a formula in prenex max ANF. %

By McNaughton's theorem \cite[Thm.~1]{McNaughton1951}, for any integers $a$ and $b_0,\dots,b_{n-1}$, the function $M(\xbar;a,\bbar) = \min(\max(a+b_0x_0+b_1x_1+\dots+b_{n-1}x_{n-1},0),1)$ can be expressed using the connectives of \L ukasiewicz logic.

Let $\varphi \leq r$ be a restricted closed condition. We may assume without loss of generality that $r = 0$. By the above discussion we can rewrite $\varphi$ as an equivalent prenex max ANF formula
\[
\psi(\xbar) = \qqq_{x_0}\qqq_{x_1}\qqq_{x_2}\dots \max_{n<N}\min_{m<M_n}\left( a_{nm} + \sum_{k < K_{nm}}b_{nmk}\chi_{nmk} \right)
\]
where each $\qqq$ is either $\inf$ or $\sup$ and each $\chi_{nmk}$ is an atomic formula. Let $\ell$ some number larger than the denominators of the coefficients in $\psi$. Consider now the formula $\psi^\dagger(\xbar)$ defined as
\[
  \qqq_{x_0}\qqq_{x_1}\qqq_{x_2}\dots \max_{n<N}\min_{m<M_n}M(\chi_{n,m,0},\dots,\chi_{n,m,K_{nm}-1};\ell!\cdot a_{nm},\ell!\cdot b_{n,m,0},\dots,\ell!\cdot b_{n,m,K_{nm}-1}).
\]
Note that $\psi^\dagger$ is equivalent to $\min(\max(\ell!\cdot \psi,0),1)$. We clearly have that in any structure $M$, $\psi(\abar) \leq 0$ if and only if $\psi^\dagger(\abar) \leq 0$, but this latter condition can be expressed in \L ukasiewicz logic by McNaughton's theorem. In particular, if we interpret $M(-;\ell!\cdot a_{nm},\ell!\cdot b_{n,m,0},\dots,\ell!\cdot b_{n,m,K_{nm}-1})$ as an expression in \L ukasiewicz logic, then we have that $M \models \psi(\abar) \geq 0$ if and only if $M$ satisfies
\[
  \neg Qx_0Qx_1Qx_3\dots \bigvee_{n<N}\bigwedge_{n<M_n}M(\chi_{n,m,0},\dots,\chi_{n,m,K_{nm}-1};\ell!\cdot a_{nm},\ell!\cdot b_{n,m,0},\dots,\ell!\cdot b_{n,m,K_{nm}-1}).
\]
where $Qx_i$ is $\exists x_i$ if $\qqq_{x_i}$ is $\sup_{x_i}$ and $\forall x_i$ if $\qqq_{x_i}$ is $\inf_{x_i}$.

There are some minor differences in the treatment of equality (i.e., the metric) in continuous logic and \L ukasiewicz logic and, relatedly, the intended semantics of continuous logic is more specific than that of \L ukasiewicz logic, but for structures without equality (i.e., general structures such as our $\Lc_e$-structures) there is no difference in expressive power.

While the above discussion is sufficient to prove that it exists, we will now give an explicit axiomatization of $\MSE$ is \L ukasiewicz logic. For the sake of compatibility with the existing \L ukasiewicz logic literature, we will switch to the convention of regarding $1$ as true. As such, we will write $x \rin y$ for $1-e(x,y)$. %
We will write $A \to B$ for connective $1-(A\dotdiv B) = \min(1-A+B,1)$ and $\bot$ for $0$. Formulas are formed from $x \rin y$ using the connectives $\to$ and $\bot$ and the quantifiers $\exists$ and $\forall$. We'll write $\LL_{\rin}$ for this set of formulas, which we will regard as a subset of the set of restricted $\Lc_e$-formulas (where we interpret $\exists x$ as $\sup_x$ and $\forall x$ as $\inf_x$). For an $\Lc_e$-structure $M$, a tuple $\abar \in M$, and a formula $\varphi(\xbar) \in \LL_{\rin}$, we say that $M$ \emph{satisfies} $\varphi(\abar)$ if $M \models \varphi(\abar) = 1$. %

It is a basic fact that the connectives $\to$ and $\bot$ can be used to define the following:
 $A \toot B\coloneqq 1-|x-y|$,
 $\neg A\coloneqq 1-A$,
 $A \wedge B\coloneqq \min(A,B)$, %
 $A \vee B\coloneqq \max(A,B)$, and %
 $A \opand B\coloneqq \max(A+B-1,0)$. %

First we need to define extensional equality: We will write $x =_e y$ as shorthand for the formula $\forall z(z \rin x \toot z \rin y)$. (This is the same thing as $1-d_e(x,y)$.) With this we can now write the $\HH$-extensionality axiom as
\[
\forall x \forall y (x \rin y \toot \exists z(x=_ez \opand z \rin y \opand z \rin y)). \tag{1}
\]
It is easy to verify that this is a literal transcription of \cref{defn:H-ext-ax}.

Given any formula $\varphi \in \LL_{\rin}$, let $\#\varphi$ be the number of instances of $\rin$ in $\varphi$.\footnote{That is to say, the number of instances of $e(x,y)$ in the corresponding restricted $\Lc_e$-formula.} %
For the axiom scheme of excision, we have
\[
  \forall \ybar \exists z \forall x (x \rin z \vee (\neg\varphi\opand\neg \varphi \opand \neg \varphi))\wedge ((\underbrace{\neg x\rin z \opand \cdots \opand \neg x \rin z}_{6\cdot\#\varphi~\text{times}}) \vee(\varphi \opand \varphi \opand \varphi)) \tag{2$_\varphi$}
\]
for every formula $\varphi(x,\ybar) \in \LL_{\rin}$ that does not contain $z$ as a free variable. As this is not a literal transcription of the axiom scheme of excision given in \cref{defn:ax-sch-exc}, we need to prove that it is equivalent. 

\begin{prop}
  An $\Lc_e$-structure $M$ is $\HH$-extensional and satisfies $\Lc_e$-excision if and only if it satisfies (1) and (2$_\varphi$) for all $\varphi \in \LL_{\rin}$.
\end{prop}
\begin{proof}
  We clearly have that $M$ is $\HH$-extensional if and only if it satisfies (1). Therefore we may assume without loss of generality that $(M,d_e)$ is a complete metric space.

  An easy inductive argument shows that for any $\varphi \in \LL_{\rin}$, $v(\varphi) = \#\varphi$. In particular, any such $\varphi$ is $(2\cdot \#\varphi)$-Lipschitz relative to $d_e$.

  For the $\To$ direction, by \cref{prop:MSE-CL-char}, $(M,d_e,\sqin_e)\models \MSE$. Fix a formula $\varphi(x,\ybar) \in \LL_{\rin}$ and a tuple $\abar \in M$. If $\#\varphi = 0$, then $\varphi(x,\abar)$ is a constant that does not depend on $\abar$, so (2$_\varphi$) is witnessed either by $\varnothing^M$ or by $V^M$. If $\#\varphi > 0$, consider the set $b = \exc{x}{1-\varphi(x,\abar)}{\frac{1}{3}}{\frac{2}{3}}$. For any $c \in M$, we have that if $\varphi^M(c,\abar) \geq \frac{2}{3}$, then $(c \rin b)^M = 1$ and if $(c \rin b)^M = 1$, then $\varphi^M(c,\abar) > \frac{1}{3}$. In particular, this implies that if $\varphi^M(f,\abar) \leq \frac{1}{3}$, then $(f\rin b)^M \leq 1-\frac{1}{6\cdot \#\varphi}$ (which implies that $M$ satisfies $\neg f \rin b \opand \cdots \opand \neg f \rin b$ with $6\cdot\#\varphi$ instances of $\neg f \rin b$). Furthermore $M$ satisfies
  \[
\forall x (x \rin b \vee (\neg \varphi(x,\abar)\opand \neg \varphi(x,\abar) \opand \neg \varphi(x,\abar)))
\]
and
\[
  \forall x (\underbrace{\neg x\rin b \opand \cdots \opand \neg x \rin b}_{6\cdot\#\varphi~\text{times}}) \vee(\varphi(x,\abar) \opand \varphi(x,\abar) \opand \varphi(x,\abar)).
\]
Since we can do this for any $\varphi \in \LL_{\rin}$ and any $\abar \in M$, we have that $M$ satisfies (2$_\varphi$) for all $\varphi \in \LL_{\rin}$.

For the $\Leftarrow$ direction, assume that $M$ satisfies (1) and (2$_\varphi$) for all $\varphi \in \LL_{\rin}$. Fix a restricted $\Lc_e$-formula $\varphi(x,\ybar)$. Assume without loss of generality that $\varphi(x,\ybar)$ contains an instance of the predicate $e$ and that $\varphi(x,\ybar)$ is in prenex max ANF. Pick a sufficiently large $\ell > 1$ and let $\varphi^\dagger(x,\ybar)$ be defined as above. In particular, we may think of $\varphi^\dagger(x,\ybar)$ as a formula in $\LL_{\rin}$ which has the property that for any $M$ and $a,\bbar \in M$, $(\varphi^\dagger)^M(a,\bbar) = \min(\max(\varphi^M(a,\bbar),0),1)$. 

Fix some $\abar \in M$ and $\delta > 0$ with $\delta < 1$ and apply (2$_{\neg \varphi^\dagger}$) to $\abar$ to get a $b \in M$ such that for every $x \in M$, 
\begin{itemize}
\item either $(x \rin b)^M > 1-\delta$ or $(\neg\neg \varphi^\dagger)^M(x,\abar) > \frac{1-\delta}{3}$ and
\item either $(\neg x \rin b)^M > \frac{1-\delta}{6\cdot\#\varphi^\dagger}$ or $(\neg \varphi^\dagger)^M(x,\abar) > \frac{1-\delta}{3}$.
\end{itemize}
This means that for every $x \in M$,
\begin{itemize}
\item if $\varphi(x,\abar) \leq \frac{1-\delta}{3\cdot \ell!}$, then $e(x,b) < \delta$ and
\item if $e(x,b) = 0$, then $\varphi(x,\abar) < \frac{2+\delta}{3\cdot \ell!}$.
\end{itemize}
The first of these clearly implies that if $\varphi(x,\abar) \leq -\delta$, then $e(x,z) < \delta$. Now suppose that for some $c \in M$, $e(c,b) \leq \e_\varphi- \delta$. We then have that there is a $f \in M$ with $(f \rin b)^M = 1$ such that $d(c,f)< \e_\varphi$. Since $(f\rin b)^M = 1$, we have that $e(f,b) = 0$, so $\varphi(f,\abar) < \frac{2+\delta}{3\cdot \ell!}$ and therefore
$  \varphi(c,\abar) < \varphi(f,\abar) + 2v(\varphi)d(c,f) 
                   < \frac{2+\delta}{3\cdot \ell!} + 2v(\varphi)\e_\varphi 
                    < \frac{1}{\ell!} + \frac{2v(\varphi)}{6v(\varphi)} 
                    < \frac{1}{2} + \frac{1}{3} < 1 < 1 + \delta.$ Since we can do this for any sufficiently small $\delta > 0$, we have that $M$ satisfies the excision axiom for $\varphi(x,\ybar)$. So since we can do this for any $\varphi(x,\ybar) \in \Lc_e$, we have that $M$ satisfies $\Lc_e$-excision.
\end{proof}

It is of course also possible to translate our axioms of infinity and other sentences described in this paper to \L ukasiewicz logic, but doing so is much more involved.

\bibliographystyle{plain}
\bibliography{../ref}

\begin{thebibliography}{10}

\bibitem{MTFMS}
Ita{\"i} Ben~Yaacov, Alexander Berenstein, C.~Ward Henson, and Alexander
  Usvyatsov.
\newblock {\em Model theory for metric structures}, volume~2 of {\em London
  Mathematical Society Lecture Note Series}, pages 315--427.
\newblock Cambridge University Press, 2008.

\bibitem{Caicedo2014}
Xavier Caicedo and Jos{\'{e}}~N. Iovino.
\newblock Omitting uncountable types and the strength of $[0,1]$-valued logics.
\newblock {\em Annals of Pure and Applied Logic}, 165(6):1169--1200, June 2014.

\bibitem{Chang1963TheAO}
C.~C. Chang.
\newblock The axiom of comprehension in infinite valued logic.
\newblock {\em Mathematica Scandinavica}, 13:9--30, 1963.

\bibitem{Fenstad-comp}
Jens~Erik Fenstad.
\newblock On the consistency of the axiom of comprehension in {the \L
  ukasiewicz} infinite valued logic.
\newblock {\em Mathematica Scandinavica}, 14(1):65--74, 1964.

\bibitem{EForster1992-EFOSTW}
T.~E. Forster.
\newblock {\em Set Theory with a Universal Set: Exploring an Untyped Universe}.
\newblock Oxford, England: Clarendon Press, 1992.

\bibitem{Forti1996}
Marco Forti and Furio Honsell.
\newblock A general construction of hyperuniverses.
\newblock {\em Theoretical Computer Science}, 156(1-2):203--215, March 1996.

\bibitem{Hjek2000}
Petr H{\'{a}}jek, Jeff Paris, and John Shepherdson.
\newblock Rational {P}avelka predicate logic is a conservative extension of
  {{\L}}ukasiewicz predicate logic.
\newblock {\em Journal of Symbolic Logic}, 65(2):669--682, June 2000.

\bibitem{HansonThesis}
James Hanson.
\newblock {\em Definability and categoricity in continuous logic}.
\newblock PhD thesis, University of Wisconsin--Madison, 2020.

\bibitem{Hanson2021AnalogR}
James Hanson.
\newblock Analog reducibility.
\newblock {\em J. Log. Comput.}, 31:1561--1597, 2021.

\bibitem{MetSpaUniv}
James Hanson.
\newblock Metric spaces are universal for bi-interpretation with metric
  structures.
\newblock {\em Annals of Pure and Applied Logic}, page 103204, 2022.

\bibitem{sep-settheory-alternative}
M.~Randall Holmes.
\newblock {Alternative Axiomatic Set Theories}.
\newblock In Edward~N. Zalta, editor, {\em The {Stanford} Encyclopedia of
  Philosophy}. Metaphysics Research Lab, Stanford University, {W}inter 2021
  edition, 2021.

\bibitem{2005.11851}
H.~Jerome Keisler.
\newblock Model theory for real-valued structures.
\newblock {\em arXiv e-prints}, page arXiv:2005.11851, 2020.

\bibitem{malitz_1976}
R.~J. Malitz.
\newblock {\em Set theory in which the axiom of foundation fails}.
\newblock PhD thesis, University of California, 1976.

\bibitem{McNaughton1951}
Robert McNaughton.
\newblock A theorem about infinite-valued sentential logic.
\newblock {\em Journal of Symbolic Logic}, 16(1):1--13, March 1951.

\bibitem{Hajek-2015}
Franco Montagna, editor.
\newblock {\em Petr H{\'{a}}jek on Mathematical Fuzzy Logic}.
\newblock Springer International Publishing, 2015.

\bibitem{Skolem1957-SKOBZK-3}
Thoralf Skolem.
\newblock Bemerkungen zum komprehensionsaxiom. dem andenken an heinrich scholz
  gewidmet.
\newblock {\em Zeitschrift fur mathematische Logik und Grundlagen der
  Mathematik}, 3(1-5):1--17, 1957.

\bibitem{tent_ziegler_2012}
Katrin Tent and Martin Ziegler.
\newblock {\em A Course in Model Theory}.
\newblock Lecture Notes in Logic. Cambridge University Press, 2012.

\bibitem{Terui-Error}
Kazushige Terui.
\newblock A flaw in {R.B.} {W}hite's article \enquote{The consistency of the
  axiom of comprehension in the infinite-valued predicate logic of \L
  ukasiewicz}.
\newblock 2014.
\newblock Unpublished.

\bibitem{weydert_thesis}
Emil Weydert.
\newblock {\em How to approximate the naive comprehension scheme inside of
  classical logic}.
\newblock PhD thesis, University of Bonn, 1989.

\bibitem{White1979}
Richard~B. White.
\newblock The consistency of the axiom of comprehension in the infinite-valued
  predicate logic of {{\L}}ukasiewicz.
\newblock {\em Journal of Philosophical Logic}, 8(1), January 1979.

\end{thebibliography}

\end{document}